\documentclass[11pt,leqno]{amsart}
\usepackage{amsthm,amsfonts,amssymb,amsmath,oldgerm}
\numberwithin{equation}{section}
\usepackage{fullpage}
\usepackage{color}
\usepackage{calrsfs}
\usepackage{bbm}
\DeclareMathAlphabet{\pazocal}{OMS}{zplm}{m}{n}
\usepackage{epsfig}
\usepackage{subfigure}

\renewcommand\d{\partial}

\newcommand\R{\mathbb R}
\newcommand\C{\mathbb C}

\newcommand\br{\begin{remark}}
\newcommand\er{\end{remark}}
\newcommand\bp{\begin{pmatrix}}
\newcommand\ep{\end{pmatrix}}
\newcommand{\be}{\begin{equation}}
\newcommand{\ee}{\end{equation}}
\newcommand\ba{\begin{equation}\begin{aligned}}
\newcommand\ea{\end{aligned}\end{equation}}

\newcommand{\bqs}{\begin{equation*}}
\newcommand{\eqs}{\end{equation*}}
\renewcommand{\Re}{\mathrm{Re}}
\renewcommand{\Im}{\mathrm{Im}}
\newcommand{\mbi}{\mathbf{i}}
\newcommand{\md}{\mathrm{d}}

\newcommand{\bap}{\begin{app}}
\newcommand{\eap}{\end{app}}
\newcommand{\begs}{\begin{exams}}
\newcommand{\eegs}{\end{exams}}
\newcommand{\beg}{\begin{example}}
\newcommand{\eeg}{\end{exaplem}}
\newcommand{\bpr}{\begin{proposition}}
\newcommand{\epr}{\end{proposition}}
\newcommand{\bt}{\begin{theorem}}
\newcommand{\et}{\end{theorem}}
\newcommand{\bc}{\begin{corollary}}
\newcommand{\ec}{\end{corollary}}
\newcommand{\bl}{\begin{lemma}}
\newcommand{\el}{\end{lemma}}
\newcommand{\bd}{\begin{definition}}
\newcommand{\ed}{\end{definition}}
\newcommand{\brs}{\begin{remarks}}
\newcommand{\ers}{\end{remarks}}

\newcommand{\NN}{{\mathbb N}}
\newcommand{\ZZ}{{\mathbb Z}}

\newtheorem{theorem}{Theorem}[section]
\newtheorem{proposition}[theorem]{Proposition}
\newtheorem{corollary}[theorem]{Corollary}
\newtheorem{lemma}[theorem]{Lemma}

\theoremstyle{remark}
\newtheorem{remark}[theorem]{Remark}
\theoremstyle{definition}
\newtheorem{definition}[theorem]{Definition}

\newtheorem{example}[theorem]{Example}

\newcommand\cA{{\mathcal A}}
\newcommand\cB{{\mathcal B}}

\newcommand\cC{{\mathcal C}}

\newcommand\cL{{\mathcal L}}

\newcommand\cQ{{\mathcal Q}}
\newcommand\cO{{\mathcal O}}

\newcommand\cZ{{\mathcal Z}}

\newcommand{\tpsi}{{\widetilde\psi}}

\newcommand{\beq}{\begin{equation}}
\newcommand{\eeq}{\end{equation}}

\title{Existence and stability of nonmonotone hydraulic shocks for the Saint 
Venant equations of inclined thin-film flow}
\author{Gr\'egory Faye$^\star$}
\address{Institut de Math\'ematiques de Toulouse, UMR5219, Universit\'e de Toulouse, CNRS UPS IMT, F-31062 Toulouse Cedex 9 France}
\email{\tt gregory.faye@math.univ-toulouse.fr}
\thanks{$^\star$Corresponding author}
\thanks{G.F. acknowledges support from the ANR via the project Indyana under grant agreement ANR- 21- CE40-0008 and Labex CIMI under grant agreement ANR-11-LABX-0040.}

\author{L.~Miguel Rodrigues}
\address{
Univ Rennes \& IUF, CNRS, IRMAR - UMR 6625, F-35000 Rennes, France}
\email{{\tt luis-miguel.rodrigues@univ-rennes1.fr}}
\thanks{Research of L.M.R. was partially supported by EPSRC grant no EP/R014604/1.}

\author{Zhao Yang}
\address{Academy of Mathematics and Systems Science, Chinese Academy of Sciences, Beijing
100190 China.}
\email{yangzhao@amss.ac.cn}
\thanks{Research of Z.Y. was partially supported by an IU COAS dissertation year fellowship}

\author{Kevin Zumbrun}
\address{Indiana University, Bloomington, IN 47405}
\email{kzumbrun@indiana.edu}
\thanks{Research of K.Z. was partially supported
under NSF grants no. DMS-1400555 and DMS-1700279}

\begin{document}

\begin{abstract}
Extending work of Yang-Zumbrun for the hydrodynamically stable case of Froude number $F<2$,
we categorize completely the existence and convective stability of hydraulic shock profiles of the Saint Venant equations of inclined thin-film flow. Moreover, we confirm by numerical experiment that asymptotic dynamics for general Riemann data is given in the hydrodynamic instability regime by either stable hydraulic shock waves, or a pattern consisting of an invading roll wave front separated by a finite terminating Lax shock from a constant state at plus infinity. Notably, profiles, and existence and stability diagrams are all rigorously obtained by mathematical analysis and explicit calculation.

\smallskip


\smallskip


\end{abstract}

\maketitle

\tableofcontents

\section{Introduction}\label{s:intro}

In \cite{YZ,SYZ}, there was carried out a comprehensive study of existence and nonlinear stability of hydraulic shock profiles for the Saint Venant equations (SV) of inclined thin film flow, under the assumption of hydrodynamic stability (or stability of constant solutions) of their endstates, a necessary condition for stability of shock profiles in standard Sobolev norms. It was shown under this condition that all profiles are {\it monotone decreasing} and {\it nonlinearly stable}. Notably, this conclusion includes both smooth and discontinuous (``subshock'' containing) profiles. 

In this paper, motivated by studies \cite{RYZ1,RYZ2} of the closely related Richard--Gavrilyuk model (RG) for inclined thin film flow, in which nonmonotone profiles, and profiles with hydrodynamically unstable endstates, play a prominent role in asymptotic behavior, we revisit this problem in more detail, seeking nonmonotone profiles in the hydrodynamically unstable case. These of course cannot be stable in standard Sobolev norms, but as seen in \cite{RYZ1}, they can nonetheless be {\it convectively stable}, or stable in an appropriately exponentially-weighted norm, hence relevant to time-asymptotic behavior. Interestingly, we do find such waves, in a case that was neglected\footnote{It was incorrectly stated there, as a side remark, that for hydrodynamically unstable endstates, the only hydraulic shock profiles were smooth, ``reverse''-direction shocks not connected to equilibrium dynamics.} in \cite{YZ}, and they appear to be convectively stable over a certain, computable regime.


The above observations have motivated the development of convective counterparts \cite{GR,FR2} to general results converting spectral stablity into linear and nonlinear stability results \cite{DR1,FR1}. Specializing \cite{FR2} to the present case, we are able to supplement our complete spectral classification with corresponding nonlinear stability results.

The inviscid Saint-Venant equations in Eulerian, nondimensionalized form appear as
\ba \label{sv}
\d_th+\d_xq&=0,\\
\d_tq+\d_{x}\left(\frac{q^2}{h}+\frac{h^2}{2F^2}\right)&=h-\frac{|q|q}{h^2},
\ea
where $h$ is fluid height; $q=hu$ is total flow, with $u$ fluid velocity; and $F>0$ is the {\it Froude number}, a nondimensional parameter depending on reference height/velocity and inclination.

These form a $2\times 2$ {\it relaxation system}, with associated formal equilibrium equation
\be\label{CE}
\d_t h + \d_x  q_*(h)=0, \qquad q_*(h):=h^{3/2},
\ee
where $q_*$ is determined by the equilibrium condition that the second component of the righthand side of \eqref{sv} vanish. The first-order, principal part of \eqref{sv}, meanwhile, coincides with the {\it equations of isentropic $\gamma$-law gas dynamics} with $\gamma =2$ \cite{Bre}. System \eqref{sv} admits constant solutions in the form of {\it equilibria} $(h,q)=(h_0,q_*(h_0))$. Stability of constant solutions, known as {\it hydrodynamic stability}, is equivalent for $2\times 2$ relaxation sytems to
the {\it subcharacteristic condition} that the equlibrium characteristic $q_*'(h_0)$ of \eqref{CE} lies between the  characteristic speeds of \eqref{sv}, yielding the classical condition of Jeffreys \cite{Je},
\be\label{hydrostab}
F<2.
\ee
Note the very special property that the condition does not depend on the particular value of $h_0$. For further discussion, see \cite{JNRYZ,YZ} on \eqref{sv} and \cite{BJRZ,RZ,BJNRZ,BJNRZ-meca} on its viscous counterpart.

In the hydrodynamically stable regime $F<2$, one does expect persistent asymptotically-constant traveling wave solutions 
\be\label{prof}
(h,q)(t,x)= (H,Q)(x-ct), \quad \lim_{z\to - \infty}(H,Q)(z)= (H_L,Q_L), \; \lim_{z\to - \infty}(H,Q)(z)= (H_R,Q_R), 
\ee
analogous to shock waves of \eqref{CE}, known as {\it relaxation shocks}, or relaxation profiles, as verified in \cite{YZ}. However, the hydrodynamically unstable regime 
$F>2$
is also of interest, in both the convectively stable regime, since this is compatible with the description of large-time dynamics arising from compactly supported perturbations of Riemann data, and, in any case, as a scenario for complex behavior and pattern formation \cite{BJRZ,RYZ1}, 
with profiles \eqref{prof} serving as potential building blocks for more complicated patterns. 
Here, we carry out an exhaustive study of existence and convective stability of hydraulic (SV) shocks
for general $F$, including both cases $F\gtrless 2$.

\subsection{Results}
We now briefly state our main results, to be expanded in the remainder of the paper.

\medskip

\subsubsection{Existence} (Section \ref{s:existence})
Expanding on the results of \cite{YZ} for $F<2$, we categorize in Proposition \ref{existprop}
all possible types of possible hydraulic shocks: namely, the three monotone types (i), (iv), and (v) 
noted in \cite{YZ}, together with two new nonmonotone types (ii) and (iii) arising for $F>2$.
These are displayed graphically in the left and right panels of Figure \ref{existence:fig},
the left one organized by the parameter $H_R/H_L$ used in \cite{YZ} and
the right one by a new, more convenient parameter $\nu_0:=\sqrt{H_{max}/H_{min}}$ in which the figures are more clear. Here $H_L$ and $H_R$ refer to the left and right limiting heights of the traveling
wave, and $H_{max}$ and $H_{min}$ to the maximum and minimum heights.
We note that type (ii)-(v) waves connect equilibria $(H_L,H_R)$ corresponding to shocks of the scalar
equilibrium system \eqref{CE}, whereas type (i) waves are smooth, monotone increasing in height, and
connect $(H_L,H_R)$ in the direction of a ``reverse shock'' of \eqref{CE}.
The former, ``forward-equilibrium shocks'' exist precisely for
\be\label{econd}
\nu:=\sqrt{{H_L}/{H_R}}>1, \quad F<\nu(\nu+1).
\ee

\subsubsection{Spectral stability} (Sections \ref{s:spectral} and \ref{s:sl})
In Section \ref{s:spectral}, we investigate stability of essential spectra in the class of scalar weighted norms, or, equivalently here, stability of absolute spectra.  We show that this fails for type (i) waves, corresponding to ``reverse'' equilibrium shocks, but is satisfied for type (ii)-(v) waves under condition
\be\label{abscond}
F < \sqrt{2\nu(\nu+1)}
\ee
(always satisfied for cases (iii)-(v)) slightly stronger than the existence condition \eqref{econd}. Indeed, as noted in Remark \ref{rk:analyticity-failure}, essential stability fails for type (i) waves and for type (ii) waves failing to satisfy \eqref{abscond} in
{\it any} type of weighted norm, yielding a conclusive result of convective instability for these types.

In Section \ref{s:sl}, we study stability of point spectrum for the remaining cases (ii), \eqref{abscond} and (iii)-(v), extending the generalized Sturm-Liouville argument introduced in \cite{YZ,SYZ} for the treatment of cases (iv)-(v). Remarkably, we are able to rigorously verify stability of point spectrum whenever the essential stability condition \eqref{abscond} is satisfied. Taken together, these results completely characterize spectral stability of hydraulic shocks of all types.
The results are displayed graphically in the panels of Figure \ref{stability:fig}.

\subsubsection{Linear and Nonlinear stability} (Section \ref{s:stab})
In Section \ref{s:stab}, we investigate for spectrally stable waves the questions of linear and nonlinear stability,
providing 
a result of convective asymptotic time-exponential orbital stability, or convergence to a 
translate of the original traveling wave. This implies in particular, time-exponential stability under localized (e.g., Gaussian- or compact-support) perturbations, a result that is new even for the $F<2$ case considered in \cite{YZ}. We have chosen here to derive these results by specializing to \eqref{sv} the general theory from \cite{FR2}. Despite the fact that analyzing directly \eqref{sv} would come with significant simplifications due to the special structure of systems of two equations compared to general systems, a detailed analysis would still be rather technical and long, without conveying much specific insight about the dynamics at hand. We stress moreover that, though the results of \cite{YZ} do not apply to the present case, their proof does contain all the main ingredients to yield the nonlinear stability of interest. Again, though a simpler form of the arguments of \cite{YZ} would be sufficient here, since time-exponential decay is simpler to handle than time-algebraic decay, a self-contained exposition of such an adaptation would still be rather long and technical.

\subsubsection{Global time-asymptotic dynamics} (Section \ref{s:num})
Finally, in Section \ref{s:num}, we carry out using CLAWPACK \cite{C1,C2} numerical experiments with (perturbed and unperturbed) ``Riemann'' or ``dambreak'' data consisting of constant equilibrium states to either side of an initial jump discontinuity, testing the ``real life'' validity of our rigorous existence/stability conclusions, in the sense of large-amplitude perturbations and resulting time-asymptotic behavior, or ``generalized Riemann solution''. We see that our analytically derived stability conditions indeed predict not only small-perturbation stability or instability, but large-scale asymptotic behavior.  Specifically, when stability holds, the asymptotic response to even large-scale localized perturbations is convergence to a hydraulic shock, monotone or nonmonotone as the case may be.  

When stability fails, on the other hand- recall, through instability of {\it essential spectrum}, having to do with convective stability of the constant left endstate of the shock- we see bifurcation to an ``invading front'' connecting roll wave patterns on the left to a constant state on the right: that is, an ``essential bifurcation'' such has been studied for smooth waves of reaction-diffusion systems in \cite{SS}. That is, our (local) stability conditions indeed successfully predict large-scale asymptotic behavior. Interestingly enough, in all our experiments the expanding speed of the instability pattern is well-predicted by the heuritics of \cite{FHSS}.

\subsection{Discussion and open problems}\label{s:disc}
The Saint Venant model has proven remarkably amenable to analysis, admitting complete solutions to both existence and stability questions now in a variety of settings. The present analysis fits among this list, giving complete and definitive answers to the questions of existence and convective stability of hydraulic shock solutions. In particular, the fact that absolute and point spectral stability could be completely characterized is quite remarkable and apparently special to (SV). It is a very interesting open problem to what extent the Sturm Liouville arguments used here might extend to large-amplitude traveling waves of general $2\times 2$ relaxation systems under a condition
of convectively stable essential spectrum, generalizing the treatment by Liu \cite{L} of small-amplitude waves in the hydrodynamically stable case.

The analyses of linear and nonlinear stability in \cite{YZ} also rely in places on specific 
computations for (SV). However, different from the situation as regards spectral stability,
the strategies for converting spectral to nonlinear stability are rather general, and
could be expected under appropriate structural conditions to carry over to the general case of relaxation models. These considerations motivate a more general and systematic study of such problems, as done by the first two authors for exponentially spectrally stable Riemann shocks \cite{FR1}, and will be the object of a future publication \cite{FR2} from which we already borrow some results.

Jointly with \cite{JNRYZ,SYZ,YZ}, the present contribution provides for (SV) an almost complete classification of traveling waves from the point of view of existence and spectral stability. Nevertheless we would like to point out that even in (SV), at the spectral level, a stability classification of waves that have characteristic points but are not periodic is still missing. Likewise one of the outstanding remaining puzzle in the nonlinear stability of relaxation waves, either smooth or discontinuous, is the treatment of waves with characteristic points, generalizing to the system case the scalar analysis of \cite{DR2}. At the nonlinear level, the corresponding difficulties are expected to occur also in the analysis of the dynamics near roll waves, which has not been touched even for (SV); see for example the discussions of \cite{JNRYZ}. Indeed, there are some additional difficulties for (SV) due to an infinite-dimensional center manifold coming from degeneracy of the model \cite{JNRYZ}. We find this to be the main open problem in the theory of general (including periodic) traveling waves. 

Finally, we mention as an interesting direction for further investigation, the rigorous treatment of the phenomenon of essential bifurcation/invading roll wave fronts that we see in our numerical experiments, the lack of smoothness and parabolic smoothing making this a nonstandard problem not covered by the methods of \cite{SS} and related references.

\section{Existence of traveling waves}\label{s:existence}

In this section, we recover the basic existence theory from \cite[Prop. 1.1]{YZ}, in the process unraveling the nonmonotone case omitted there. We focus on traveling waves with piecewise smooth profiles without characteristic point, neither on profiles nor at infinity. The presence of characteristic points is expected to have dramatic effects on the existence, spectral stability and nonlinear dynamics; see the related analyses in \cite{JNRYZ,DR2}. We also restrict to waves with nonnegative velocities so that absolute values may be dropped, but one may be careful to check as a consistency condition that indeed $Q\geq0$.

To expect some form of uniqueness when dealing with discontinuous solutions we need to impose some form of entropy conditions. Combined with the non-characteristic assumption, even the weaker forms of the latter imply that the traveling wave profiles exhibit at most one discontinuity. We again refer to \cite{JNRYZ,DR2} for a detailed discussion. Without loss of generality, by translational invariance, the discontinuity of wave profiles may be fixed at $x=0$. When restricting further to asymptotically constant profiles, they also yield that limiting endstates are distinct.

Here and elsewhere, let $[h]_x:=h(x^+)-h(x^-)$ of a quantity $h$ across a discontinuity located at $x$, and $[h]:=[h]_0$. In smooth regions, traveling-wave solutions \eqref{prof} satisfy
\begin{align}
-cH'+Q'&=0,&-cQ'+\left(\frac{Q^2}{H}+\frac{H^2}{2F^2} \right)'&=H-\frac{|Q|Q}{H^2},&
\label{TW}
\end{align}
whereas at discontinuities, we have the Rankine-Hugoniot conditions
\begin{align}
-c\left[H\right]+\left[Q\right]&=0,&-c\left[Q\right]+\left[\frac{Q^2}{H}+\frac{H^2}{2F^2} \right]&=0.&
\label{RaHu}
\end{align}
A simple observation is that the end states $(H_R,Q_R)$ and $(H_L,Q_L)$ of the traveling wave profiles \eqref{prof} must be equilibria, that is $Q_{L,R}=q_*(H_{L,R})=H_{L,R}^{3/2}$ (since we are working in the physical range $H>0$). Combined together the first halves of \eqref{TW} and \eqref{RaHu} are equivalent to the existence of a constant $q_0$ such that
\be
cH-Q\equiv q_0\,.
\label{conservative}
\ee
With such a $q_0$ fixed, the second equation of \eqref{TW} leads to the scalar ODE
\be
\left(-\frac{q_0^2}{H^2}+\frac{H}{F^2}\right)H'=\frac{H^3-(cH-q_0)^2}{H^2},
\label{ODEint}
\ee
while the second Rankine-Hugoniot condition in \eqref{RaHu} reads
\be
\left[\frac{q_0^2}{H}+\frac{H^2}{2F^2} \right]=0.
\label{RaHubis}
\ee

Equation \eqref{ODEint} is a scalar first-order ODE, so that it cannot connect smoothly an endstate to itself (in a non stationary way). We have already discussed that when instead a discontinuity is indeed present, we must have $(H_R,Q_R)\neq(H_L,Q_L)$ so that in any case $H_L\neq H_R$. Therefore from \eqref{conservative} stems that $(H_L,H_R, c,q_0)$ must satisfy 
\begin{align*}
c&=\frac{q_*(H_L)-q_*(H_R)}{H_L-H_R}=\frac{H_L+\sqrt{H_LH_R}+H_R}{\sqrt{H_L}+\sqrt{H_R}},&
q_0&=\frac{H_LH_R}{\sqrt{H_L}+\sqrt{H_R}}\,,
\end{align*}
 and necessarily $c>0$, $q_0>0$. Note that then the condition $Q\geq 0$ becomes $H\geq q_0/c$ with
\[
 \frac{q_0}{c}
=\frac{H_LH_R}{H_L+\sqrt{H_LH_R}+H_R}<\min(\{H_L,H_R\})\,.
\]
Moreover, from the sign of $q_0$ and entropy conditions stem, when a discontinuity is present,
\begin{align*}
-\frac{q_0}{H_L}+\sqrt{\frac{H_L}{F^2}}&>0>-\frac{q_0}{H_R}+\sqrt{\frac{H_R}{F^2}}\,,
\end{align*}
which is equivalently written as
\begin{align*}
H_L&>H_s>H_R\,,&
H_s&:=\left(q_0\,F\right)^{\frac23}\,.
\end{align*}

The scalar ODE \eqref{ODEint} may be factorized as
\be
\label{profileODE}
H'=\frac{F^2 \left(H - H_L\right) \left(H - H_R\right) \left(H-H_{out}\right)}{(H-H_s)(H^2+HH_s+H_s^2)},
\ee
where 
\[
H_{out}:=\frac{H_LH_R}{(\sqrt{H_L}+\sqrt{H_R})^2}
=\frac{H_LH_R}{H_L+2\sqrt{H_LH_R}+H_R}\,.
\]
Note that in any case $H_{out}<q_0/c<\min(\{H_L,H_R\})$ and recall that solutions to \eqref{profileODE} taking values below $q_0/c$ have no significance for the original traveling wave profile problem. Therefore for the discontinuous profiles, one needs $H_{out}<H_R<H_s<H_L$ and a simple one-dimensional phase-portrait analysis shows that the piece converging to $H_R$ must be constant. As a consequence, in terms of $H_*=H(0^-)$, \eqref{RaHubis} is reduced to
\[
\frac{q_0^2}{H_*}+\frac{H_*^2}{2F^2}=\frac{q_0^2}{H_R}+\frac{H_R^2}{2F^2}
\]
which possesses a unique positive solution distinct from $H_R$
\[
H_*:=\sqrt{2\frac{H_s^3}{H_R}+\frac{H_R^2}{4}}-\frac{H_R}{2}\,.
\]
Note that $H_s>H_R$ implies $H_*>H_R$.

For the sake of comparison with \cite{YZ}, let us introduce the scaling parameter
\[
\nu:=\sqrt{\frac{H_L}{H_R}}
\]
and express the above quantities as
\begin{align*}
c&=\frac{\nu^2+\nu+1}{\nu+1}\sqrt{H_R}\,,&
q_0&=\frac{\nu^2}{\nu+1}H_R^{\frac32}\,,&
H_s&=\left(\frac{F\nu^2}{\nu+1}\right)^{\frac{2}{3}}H_R\,,&
\end{align*}
and
\begin{align*}
H_{out}&=\frac{\nu^2}{\nu^2+2\nu+1}H_R\,,&
H_*&=\frac{-(\nu+1)+\sqrt{8F^2\nu^4+\nu^2+2\nu+1}}{2\left(\nu+1\right)}H_R\,.
\end{align*}

We have the following extension/correction of \cite[Prop. 1.1]{YZ}. Cases (ii) and (iii) were mistakenly omitted there.

\bpr\label{existprop}
Let $(H_L,H_R)$ be a couple of positive heights.\\
When $H_L<H_R$, that is when $\nu<1$, there exists only one kind of non-characteristic wave profiles connecting $H_L$ to $H_R$,
\begin{itemize}
\item [(i)] 
\emph{increasing smooth profiles}, that do exist if and only if $H_L<H_R<H_s$, that is, if and only if
\be\label{case1}
\nu<1\,,\qquad\frac{\nu+1}{\nu^2}<F.
\ee
\end{itemize}
When $H_R<H_L$, that is when $\nu>1$, there exist four kinds of non-characteristic waves connecting $H_L$ to $H_R$,
\begin{itemize}
\item [(ii)] 
\emph{nonmonotone discontinuous profiles}, consisting of a smooth portion
increasing from $H_L$ to $H_*$, connected by an entropy-admissible Lax shock to a portion constant equal to $H_R$, that do exist if and only if $H_R<H_s<H_L<H_*$, that is, if and only if
\be\label{case2}
\nu>1\,,\qquad
\frac{(\nu+1)\sqrt{2(\nu^2+1)}}{2\nu}  < F    <\nu (\nu+1);
\ee
\item[(iii)] 
\emph{Riemann profiles}, consisting of a portion equal to $H_L$, connected by an entropy-admissible Lax shock to a portion constant equal to $H_R$, that do exist if and only if $H_R<H_s<H_*=H_L$, that is, if and only if
\be\label{case3}
\nu>1\,,\qquad
F=\frac{(\nu+1)\sqrt{2(\nu^2+1)}}{2\nu};
\ee
\item[(iv)] 
\emph{decreasing discontinuous profiles}, consisting of a smooth portion
decreasing from $H_L$ to $H_*$, connected by an entropy-admissible Lax shock to a portion constant equal to $H_R$, that do exist if and only if $H_R<H_s<H_*<H_L$, that is, if and only if
\be\label{case4}
\nu>1\,,\qquad
\frac{\nu+1}{\nu^2}<F<\frac{(\nu+1)\sqrt{2(\nu^2+1)}}{2\nu};
\ee
\item[(v)] 
\emph{smooth decreasing profiles}, that do exist if and only if $H_s<H_R<H_L$, that is, if and only if
\be\label{case5}
\nu>1\,,\qquad
F<\frac{\nu+1}{\nu^2}.
\ee
\end{itemize}
\epr

\begin{proof}
Simple one-dimensional phase-portrait considerations provide the classification in terms of respective positions of $H_L$, $H_R$, $H_s$ and $H_*$, that may be readily translated as conditions on $\nu$ and $F$. There only remains to point out that, in case (ii), we have used that when $\nu>1$,
\[
\frac{\nu+1}{\nu^2}<\frac{(\nu+1)\sqrt{2(\nu^2+1)}}{2\nu}
\]
to discard as redundant one of the inequalities. Incidentally we also point out that when $\nu>1$ 
\[
\frac{(\nu+1)\sqrt{2(\nu^2+1)}}{2\nu}<\nu (\nu+1)\,,
\]
so that case (ii) is indeed non empty. This completes the proof.
\end{proof}
\begin{figure}[htbp]
    \centering
    \includegraphics[scale=0.335]{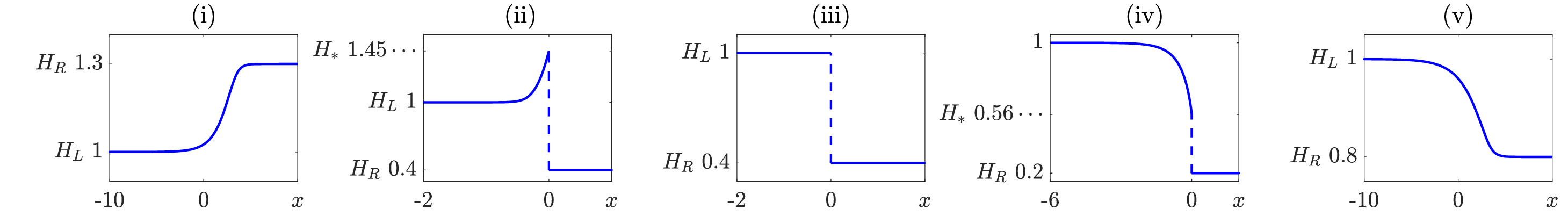}
    \caption{Examples of cases (i)--(v) of hydraulic shock profiles prescribed in Proposition~\ref{existprop}. (i) an increasing smooth profile with $H_L=1$, $H_R=1.3$, $F=3$;  (ii) a nonmonotone discontinuous profile with $H_L=1$, $H_R=0.4$, $F=3$; (iii) a Riemann profile with $H_L=1$, $H_R=0.4$, $F=\tfrac{\sqrt{7}}{2}+\sqrt{\tfrac{7}{10}}$; (iv) a decreasing discontinuous profile with $H_L=1$, $H_R=0.2$, $F=1.5$; (v) a smooth decreasing profile with $H_L=1$, $H_R=0.8$, $F=1.5$.}
    \label{fig:profile}
\end{figure}
\br\label{hydrormk}
With hydrodynamical stability in mind, let us compare different $\nu$-dependent Froude thresholds to the critical value $2$. For any $\nu<1$, 
\[
\frac{\nu+1}{\nu^2}=\frac{1}{\nu}+\frac{1}{\nu^2}>2
\]
so that case (i) is always hydrodynamically unstable. When $\nu>1$, 
\begin{align*}
\frac{\nu+1}{\nu^2}=\frac{1}{\nu}+\frac{1}{\nu^2}&<2\,,&
\frac{(\nu+1)\sqrt{2(\nu^2+1)}}{2\nu}&>2\,,&
\end{align*}
the latter inequality following from the fact that its left-hand side is increasing with $\nu$ and takes the value $2$ at $\nu=1$. Thus cases (ii) and (iii) are always hydrodynamically unstable, case (v) is always hydrodynamically stable and case (iv) may or may not be hydrodynamically unstable. Case (v), and case (iv) when $F<2$ have been thoroughly analyzed in \cite{SYZ,YZ}. 
\er

\br\label{rollrmk}
In the above discussion, we have decided in advance that we were looking for non-characteristic traveling waves connecting $H_L$ to $H_R$. For the convenience of the reader, we now provide a more systematic treatment of non constant waves in terms of
\begin{align*}
H_{min}&:=\min(\{H_L,H_R\})\,,&
H_{max}&:=\max(\{H_L,H_R\})\,,&
\nu_0&:=\sqrt{\frac{H_{max}}{H_{min}}}\,,&
\end{align*}
with $(H_{min},H_{max})$ now merely playing the role of wave parameters (replacing $(c,q_0)$).
\begin{enumerate}
\item When $F>\nu_0(\nu_0+1)$, only waves of case (i) exist, with $H_L=H_{min}$ and $H_R=H_{max}$. 
\item When $F=\nu_0(\nu_0+1)$, $H_s=H_{max}$ and there exist two families of traveling waves, one family with each member beginning by a smooth infinite portion arising from $H_L=H_{min}$, connected by a Lax shock to an infinite array of increasing portions passing though $H_s$, connected by Lax shocks, the family being parameterized by an arbitrary sequence of lengths taken in $(0,+\infty)^\NN$, the other family with each member consisting in an infinite\footnote{In both directions.} array of increasing portions passing though $H_s$, connected by Lax shocks, the family being parameterized by an arbitrary sequence of lengths taken in $(0,+\infty)^\ZZ$. The latter family includes periodic ``roll wave'' solutions of the type discovered by Dressler \cite{Dr}, that is, periodic traveling-wave solutions with exactly one discontinuity and one characteristic point by period. A comprehensive study of their spectral stability may be found in \cite{JNRYZ}.
\item When
\[
\frac{\nu_0+1}{\nu_0^2}  < F    <\nu_0 (\nu_0+1)
\]
only waves of cases (ii)-(iii) and (iv) exist, with $H_R=H_{min}$ and $H_L=H_{max}$. 
\item When 
\[
F=\frac{\nu_0+1}{\nu_0^2}\,,
\]
there exists no wave.
\item When 
\[
F<\frac{\nu_0+1}{\nu_0^2}\,,
\]
only waves of case (v) exist, with $H_R=H_{min}$ and $H_L=H_{max}$.
\end{enumerate}
\er
We summarize the existence results in Figure~\ref{existence:fig}.

\begin{figure}[htbp]
\begin{center}
\includegraphics[scale=0.32]{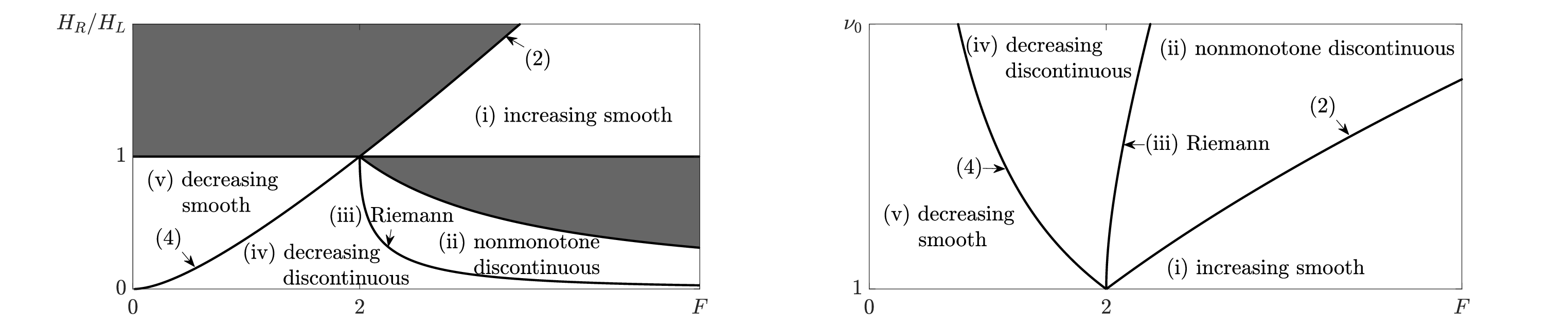}
\end{center}
\caption{Left panel: domains of cases (i)-(v) from Proposition~\ref{existprop}, extending the scope of \cite[Fig. 3. (b)]{YZ} beyond the box $0<H_R/H_L<1$, $0<F<2$ (note that, by re-scaling, $H_L$ is fixed to be $1$ in \cite{YZ}); Right panel: visualization of domains of cases (i)-(v) (2) (4) by incorporating $\nu_0$ from Remark~\ref{rollrmk}. }
\label{existence:fig}
\end{figure}

\section{Spectral framework and essential spectrum}\label{s:spectral}

We now turn to an examination of the spectral stability of waves listed in Proposition~\ref{existprop}. When doing so, we use extensively standard elements of spectral theory specialized to nonlinear wave stability. We give little detail on those but rather refer the reader to the already classical \cite{ZH,Z-course,MZ,Sandstede,KapitulaPromislow-stability} for detailed comprehensive exposition and to the recent \cite{Blochas-Rodrigues} for a self-contained worked-out case that could hopefully be used as a gentle entering gate. For discontinuous waves, this involves, at least implicitly, Evans-Lopatinski\u{\i} determinants, that interpolate between pure Evans functions used in smooth wave analysis and pure Lopatinski\u{\i} determinants used to analyze local-in-time persistence near shocks. On the latter we refer for instance to \cite[Section~4.6]{BenzoniGavage-Serre_multiD_hyperbolic_PDEs}. Evans-Lopatinski\u{\i} determinants are commonly encountered in the literature about spectral and linear stability of shocks; see for instance \cite{Godillon,Godillon-Lorin,Texier-Zumbrun,JNRYZ}. 

\subsection{Linearization and spectrum}

To introduce the relevant spectral problem in a concise way, let us write \eqref{sv} in standard abstract form 
\be\label{abstract}
\partial_t w + \partial_x(f(w))= r(w)\,,
\ee
with
\begin{align*}
w&:=\bp h\\ q\ep\,,&
f(w)&:=\bp q \\ \frac{q^2}{h}+\frac{h^2}{2F^2}\ep\,,& 
r(w)&:=\bp 0 \\ h-\frac{|q|q}{h^2} \ep\,.
\end{align*}
System \eqref{abstract} must be satisfied at least in weak sense, thus, for piecewise smooth solutions we impose \eqref{abstract} to hold in a strong sense on domains corresponding to smooth parts and along a jump whose location at time $t$ is at $\varphi(t)$ we impose Rankine-Hugoniot jump conditions
\be\label{RH}
\frac{\md \varphi}{\md t} [w]_\varphi =[f(w)]_\varphi\,.
\ee

Pick a non-characteristic traveling wave of profile $W:=(H,Q)$ and speed $c$. When $W$ is smooth, writing equations in terms of $v$, with $w(t,x)=W(x-ct)+v(t,x-ct)$, and replacing nonlinear terms with a source term, lead to 
\be\label{affine-smooth}
\partial_t v+ \partial_x(A\,v) = E\,v+F,
\qquad\textrm{on }\R_+\times\R
\ee
where the source term $F$ depends on space and time but the matrix-valued coefficients $A$ and $E$ depend only on $x$ and are explicitly given by
\bqs
A(x):=\left(\begin{matrix} -c & 1 \\ -\frac{Q(x)^2}{H(x)^2}+\frac{H(x)}{F^2} & -c+2\frac{Q(x)}{H(x)} \end{matrix} \right), \quad E(x):=\left(\begin{matrix} 0 & 0 \\ 1+2\frac{Q(x)^2}{H(x)^3} & -2\frac{Q(x)}{H(x)^2} \end{matrix} \right)\,.
\eqs
In turn, when $W$ possesses a discontinuity at $0$, proceeding in the same way but in terms of $(v,\psi)$, with $w(t,x)=W(x-(ct+\psi(t)))+\tilde{w}(t,x-(ct+\psi(t)))$, $v=\tilde{w}-\psi\,W'$, yields
\be\label{affine-disc}
\left\{
\begin{array}{rl}
\partial_t v+ \partial_x(A\,v)&=\ E\,v+F,\qquad\textrm{on }\R_+\times\R^*,\\
\frac{\md \psi}{\md t}\left[\,W\,\right] -\psi\left[r(W)\right]
&=\ \left[Av\right]+G,\qquad\textrm{on }\R_+,
\end{array}
\right.
\ee
where $(A,E,F)$ are as above and $G$ is a time-dependent source term. 

\br\label{sourcermk}
It is customary in smooth wave analysis to directly discard source terms. This is justified by the fact that when considering initial value problems one may recover the general source-term case through Duhamel's formula. However, for discontinuous waves, the linearized problem is a mixed initial boundary value problem and the arguments fails. The source terms $G$ that may be recovered by the Duhamel formula are those that are pointwise in time colinear with $[W]$. On a directly related note, let us observe that whereas \eqref{affine-smooth} directly fits in standard semigroup theory, \eqref{affine-disc} does not but it does belong to the class of problems that can be analyzed through the more general, infinite-dimensional Laplace transform theory, as covered in \cite{ABHN_Laplace}, and we shall extrapolate standard spectral terminology to this case.
\er

Applying the Laplace transform to the above linearized problems yields respectively
\be\label{spectral-smooth}
\lambda v+(A\,v)'=E\,v+F,\qquad\textrm{on }\R
\ee
and
\be\label{spectral-disc}
\left\{
\begin{array}{rl}
\lambda v+(A\,v)'&=\ E\,v+F,\qquad\textrm{on }\R^*,\\
\psi\left[\,\lambda\,W-r(W)\right]
&=\ \left[Av\right]+G,
\end{array}
\right.
\ee
with a different meaning for $(v,\psi,F,G)$, and $\lambda\in\C$ a spectral parameter. For the sake of concision, let us set
\begin{align*}
L_\lambda(v)&\,:=\,\lambda v+(A\,v)'-E\,v\,,\\
L_\lambda((v,\psi))&\,:=\,(\lambda v+(A\,v)'-E\,v,\psi\left[\,\lambda\,W-r(W)\right]
-\left[Av\right])\,,
\end{align*}
in respective cases.

For some choice of functional spaces $(X,Y)$, we say that $\lambda$ does not belong to the $(X,Y)$-spectrum of either \eqref{affine-smooth} or \eqref{affine-disc} if and only $L_\lambda$ is invertible as a bounded operator from $Y$ to $X$. In the smooth case, this matches the classical definition of the spectrum of the generator of the dynamics on $X$, when $Y$ is chosen to be the corresponding domain.

Consistently we say that the wave under consideration is spectrally $(X,Y)$-stable if the corresponding $(X,Y)$-spectrum is included in $\{\lambda\in\C;\Re(\lambda)\leq 0\}$ and that it is spectrally $(X,Y)$-unstable otherwise. We say that it is exponentially spectrally $(X,Y)$-stable if there exists $\theta>0$ such that the $(X,Y)$-spectrum is included in $\{\lambda\in\C;\Re(\lambda)\leq -\theta\}\cup\{0\}$ and $0$ has multiplicity $0$ if the wave is smooth and $W'\notin Y$, $1$ otherwise.

When stability/instability is considered with respect to $(X,Y)=(L^2(\R;\C^2),H^1(\R;\C^2))$ in the smooth case, or $(X,Y)=(L^2(\R^*;\C^2)\times \C,H^1(\R^*;\C^2)\times\C)$ in the discontinuous case, we drop any mention to the functional pair $(X,Y)$. This particular choice of functional spaces takes into account that our profiles are non-characteristic. From this property also stems that the spectrum is not really affected by the level of regularity encoded by functional spaces provided that they are chosen consistently. However it is strongly impacted by the level of localization. 

To take this into account, we introduce for $(\eta_L,\eta_R)\in \R^2$, the weighted spaces 
\begin{align*}
\mathcal{X}_{\eta_L,\eta_R}&(\R;\C^2)\\
&:=\left\{\,v\in\mathcal{X}_{loc}(\R;\C^2)\,|\,
(e^{-\eta_L\,\cdot}v)_{|\R_-}\in\mathcal{X}(\R_-;\C^2)\textrm{ and }
(e^{-\eta_R\,\cdot}v)_{|\R_+}\in\mathcal{X}(\R_+;\C^2)\,\right\}\\
\mathcal{X}_{\eta_L,\eta_R}&(\R^*;\C^2)\\
&:=\left\{\,v\in\mathcal{X}_{loc}(\R^*;\C^2)\,|\,
(e^{-\eta_L\,\cdot}v)_{|\R_-^*}\in\mathcal{X}(\R_-^*;\C^2)\textrm{ and }
(e^{-\eta_R\,\cdot}v)_{|\R_+^*}\in\mathcal{X}(\R_+^*;\C^2)\,\right\}\,,&
\end{align*}
with $\mathcal{X}=L^2$ or $\mathcal{X}=H^1$. Consistently, when talking about stability, we replace any mention to a pair $(X,Y)$ with the adverb \emph{convectively} if it can be achieved respectively with $(X,Y)=(L^2_{\eta_L,\eta_R}(\R;\C^2),H^1_{\eta_L,\eta_R}(\R;\C^2))$ or $(X,Y)=(L^2_{\eta_L,\eta_R}(\R^*;\C^2)\times \C,H^1_{\eta_L,\eta_R}(\R^*;\C^2)\times \C)$ for some $(\eta_L,\eta_R)$ such that $\eta_L\geq0$ and $\eta_R\leq0$. Correspondingly convective instability refers to the failure of convective stability. When it will be convenient to keep track of the chosen weights we will replace the general term ``convectively'' with the more specific term ``$(\eta_L,\eta_R)$-weightedly''.

\br\label{nonlinearrmk}
The constraint ($\eta_L\geq0$ and $\eta_R\leq0$) imposed in the definition of convective stability is motivated by the will to pave the way for nonlinear analysis. At a semi-abstract level, a functional space $Z$ appearing at the spectral level (for scalar components) is thought as a good space for nonlinear analysis if $Z\cap L^\infty$ is an algebra. This leads to the above requirements on weights. In the discontinuous case, another obstruction to a nonlinear analysis may be anticipated. Indeed in a Duhamel formulation source terms would contain terms that decay spatially like the square of components of $\psi\,W'$, which may belong to a $(\eta_L,\eta_R)$-weighted space only if\footnote{The first part corresponds to the Riemann shock case.} $W'\equiv 0$ or 
\begin{align*}
\eta_L&<2\eta_L^\infty\,,
\end{align*}
where 
\begin{align*}
\eta_L^\infty&:=\frac{F^2\left(H_L-H_R\right) \left(H_L-H_{out}\right)}{(H_L-H_s)(H_L^2+H_LH_s+H_s^2)}\,.
\end{align*}
The situation is dramatically different in the smooth case since there one needs to introduce a phase shift (which would also appear in nonlinear terms) only if $W'$ does belong to the kernel of $L_0$, that is, only if $\eta_L<\eta_L^\infty$ and $\eta_R>\eta_R^\infty$ where $\eta_L^\infty$ is as above and
\begin{align*}
\eta_R^\infty&:=\frac{F^2 \left(H_R - H_L\right)\left(H_R-H_{out}\right)}{(H_R-H_s)(H_R^2+H_RH_s+H_s^2)}\,,
\end{align*}
which do imply $\eta_L<2\eta_L^\infty$ and $\eta_R>2\eta_R^\infty$. In the discontinuous case, a phase shift is required no matter what; in the foregoing derivation of \eqref{affine-disc} we have partially hidden it when we have moved from $\tilde{w}$ to $v$. In our definition, for the sake of simplicity, we have chosen not to include the extra constraint $\eta_L<2\eta_L^\infty$ of the discontinuous case but 
as we check in Remark~\ref{follow-up-NL} it turns out that in the present case extra constraints already enforce $\eta_L<\eta_L^\infty$. 
\er

\br\label{weightsrmk}
Our current definition of convective stability/instability uses scalar exponential weights. Though this choice is the most usual one, it is also somewhat arbitrary. However, as we shall detail in Remark~\ref{rk:analyticity-failure}, in the present case, no substantial further gain in stabilization may be expected from the use of more complex weights.
\er

\subsection{Essential spectrum, consistent splitting and absolute instability}\label{s:consist}

A subset of the $(X,Y)$-spectrum is constituted of the $\lambda$ such that $L_\lambda$ is not Fredholm of index $0$ as a bounded operator from $Y$ to $X$. By analogy with the standard case, we call this part the $(X,Y)$-essential spectrum. The essential spectrum is therefore the set of $\lambda$ such that the codimension of the range of $L_\lambda$ and the dimension of its kernel are not equal, a clear obstruction to invertibility, which occurs when both are zero.

By using that being Fredholm of index $0$ is invariant by compact perturbations and that the problem at hand is non characteristic with coefficients converging exponentially fast to their limits, one may derive a characterization of the essential spectrum. We do not provide details on the proof of the latter but we refer the reader to the Appendix to \cite[Chapter~5]{He} for a worked out version in a closely related context. 

To discuss the outcome, we introduce
\begin{align*}
A_h&:=\bp -c&1\\h\,\left(-1+\frac{1}{F^2}\right)&-c+2\sqrt{h}\ep\,,&
E_h&:=\bp 0&0\\3&-\frac{2}{\sqrt{h}}\ep\,,
\end{align*}
\begin{align*}
G_h(\lambda)&:=\,A_h^{-1}\,(E_h-\lambda)
\,=\,\frac{1}{\left(c-\sqrt{h}\right)^2-\frac{h}{F^2}}
\bp -\lambda\,(-c+2\sqrt{h})-3&\frac{2}{\sqrt{h}}+\lambda\\
-\lambda\,h\,\left(1-\frac{1}{F^2}\right)-3\,c&c\,\left(\frac{2}{\sqrt{h}}+\lambda\right)\ep\,,&
\end{align*}
and recall that
\begin{align*}
c-\sqrt{h}&=\frac{q_0}{h}=\frac{1}{F}\frac{H_s^{\frac32}}{h}\,,&
\textrm{when }h=H_L,\,H_R\,.
\end{align*}
Then $\lambda$ does not belong to the $(\eta_L,\eta_R)$-weighted essential spectrum if and only if $G_{H_L}(\lambda)$ has no eigenvalue with real part $\eta_L$, $G_{H_R}(\lambda)$ has no eigenvalue with real part $\eta_R$ and the sum of the number of eigenvalues of $G_{H_L}(\lambda)$ with real part greater than $\eta_L$ and of the number of eigenvalues of $G_{H_R}(\lambda)$ with real part lesser than $\eta_R$ equals $2$ in the smooth case, $1$ in the discontinuous case. 

By continuity in $\lambda$, $(\eta_L,\eta_R)$-weighted stability requires that each of the above-mentioned numbers is constant in $\lambda$ on $\{\,\lambda\,;\,\Re(\lambda)>0\,\}$, a property referred to as \emph{consistent splitting} in part of the literature. Now, note that when $|\lambda|\to\infty$, eigenvalues of $G_h(\lambda)$ expand as 
\[
\frac{\lambda}{c-\sqrt{h}\mp\frac{\sqrt{h}}{F}}+\cO(1)
\]
which when specialized to $h=H_L$ or $H_R$ is equivalently written as
\[
\frac{\lambda}{\frac{q_0}{h}\mp\frac{\sqrt{h}}{F}}+\cO(1)
\,=\,
\frac{\lambda\,h\,F}{H_s^{\frac32}\mp h^{\frac32}}+\cO(1)\,.
\]
The leading order part of these spatial eigenvalues is given by the eigenvalues of $-\lambda\,A_h^{-1}$ and thus is directly connected to the characteristic velocities of $\d_t+A_h\,\d_x$. As a consequence, for $h=H_L$ or $H_R$, for any $\eta>0$ there exists $C_\eta>0$ such that when $|\lambda|\geq C_\eta$ and $\Re(\lambda)\geq \eta$, $G_{H_L}(\lambda)$ has two eigenvalues with positive real parts when $h<H_s$ and eigenvalues with real parts of opposite sign when $h>H_s$.

As a consequence, a specific way in which failure of convective stability (resp. absolute convective instability) may occur in the present case is when for $h=H_L$ or $H_R$ such that $h>H_s$, there exists $\lambda$ with positive real part (resp. nonzero with nonnegative real part) such that the eigenvalues of $G_h(\lambda)$ have the same real part. This scenario matches what is commonly designated in the literature as failure of extended consistent splitting or \emph{absolute instability}. To decide whether an absolute instability may indeed occur, let us first make explicit that the eigenvalues of $G_h(\lambda)$ are given as
\begin{align}\label{eq:gamma}
\gamma_{\pm,h}(\lambda)
:=\frac{1}{\left(c-\sqrt{h}\right)^2-\frac{h}{F^2}}
\,\left(\lambda\,(c-\sqrt{h})-\left(\frac32-\frac{c}{\sqrt{h}}\right)
\pm\sqrt{\cQ_h(\lambda)}\right)
\end{align}
where
\[
\cQ_h(\lambda)
:=\lambda^2\frac{h}{F^2}+\lambda\,\left(-(c-\sqrt{h})+\frac{2\sqrt{h}}{F^2}\right)
+\left(\frac32-\frac{c}{\sqrt{h}}\right)^2\,,
\]
for some determination of $\sqrt{\cQ_h(\lambda)}$. Note that $\gamma_{\pm,h}(\lambda)$ share the same real part exactly when $\cQ_h(\lambda)$ is a nonpositive real number. Since
\begin{align*}
\Re(\cQ_h(\lambda))&=-\Im(\lambda)^2\frac{h}{F^2}+\cQ_h(\Re(\lambda))\,,\\
\Im(\cQ_h(\lambda))&=\Im(\lambda)\,\left(
2\Re(\lambda)\,\frac{h}{F^2}-(c-\sqrt{h})+\frac{2\sqrt{h}}{F^2}\right)\,.
\end{align*}
one readily deduces that the latter does occur for some $\lambda$ with positive real part (resp. nonzero with nonnegative real part) if and only if 
\begin{align*}
c-\sqrt{h}&>\frac{2\sqrt{h}}{F^2}\,,&
\Big(\textrm{ resp. }c-\sqrt{h}&\geq\frac{2\sqrt{h}}{F^2}\ \Big)\,.
\end{align*}

\br\label{rk:analyticity-failure}
In the present case, when convective stability (resp. exponential convective stability) fails in the foregoing way, there also exists a $\lambda$ with positive real part (resp. nonzero with nonnegative real part) such that $\gamma_{\pm,h}(\lambda)$ are equal. At this $\lambda$, the resolvent operator cannot be continuously extended even as an operator from the space of test functions to distributions. This shows that, in the present case, such absolute instabilities cannot be cured in any sensible sense, in particular not by replacing exponential weights by a more general class of reasonable weights.
\er

In order to elucidate further a possible absolute instability, we compute that when $h=H_L$ or $h=H_R$,
\[
c-\sqrt{h}-\frac{2\sqrt{h}}{F^2}
\,=\,\sqrt{h}\left(\frac1F\frac{H_s^{\frac32}}{h^{\frac32}}-\frac{2}{F^2}\right)
\,=\,\begin{cases}
\sqrt{h}\left(\frac{\nu^2}{\nu+1}-\frac{2}{F^2}\right)&h=H_R\,,\\
\sqrt{h}\left(\frac{1}{\nu(\nu+1)}-\frac{2}{F^2}\right)&h=H_L\,.
\end{cases}
\]
Recalling that the scenario also requires $h>H_s$ and observing that when $\nu>1$, 
\begin{align*}
\frac{(\nu+1)\sqrt{2(\nu^2+1)}}{2\nu}<\sqrt{2\nu(\nu+1)}<\nu(\nu+1)
\end{align*}
one deduces that absolute instability may only occur in case (ii) of Proposition~\ref{existprop} and does occur when 
\[
\sqrt{2\nu(\nu+1)}<F<\nu(\nu+1)\,.
\]

\subsection{Smooth fronts}\label{s:smooth}

We temporarily restrict the discussion to smooth profiles, that is, to cases (i) and (v) of Proposition~\ref{existprop}.

Case (v) has already been studied in \cite[Section~3]{SYZ} with conclusion that all profiles of case (v) are spectrally stable but not exponentially spectrally stable. With a few more simple computations one may even check that this spectral stability is of diffusive type in a sense compatible with the application of general results from \cite{MaZ} and conclude to nonlinear asymptotic stability with algebraic decay rates. 

The only question left concerning case (v) is whether also holds convective exponential spectral stability. In this case the only obstacle to exponential spectral stability without weight is the presence of two curves of essential spectrum passing through $\lambda=0$ tangentially to the imaginary axis, one for each spatial infinity. Recall that since, in case (v), $H_R>H_s$ and $H_L>H_s$, there holds for $h=H_L$, $H_R$,
\begin{align*}
\Re(\gamma_{-,h}(\lambda))&>0>\Re(\gamma_{+,h}(\lambda))\,,&
\textrm{ when }\Re(\lambda)\gg 1\,.
\end{align*}
Therefore, to conclude convective exponential spectral stability, one needs only to check that curves of essential spectrum near $\lambda=0$ are due to changes of sign of $\Re(\gamma_{+,H_L}(\lambda))$ and $\Re(\gamma_{-,H_R}(\lambda))$. Since 
\[
\gamma_{\pm,h}(0)
:=\frac{1}{\left(c-\sqrt{h}\right)^2-\frac{h}{F^2}}
\,\left(-\left(\frac32-\frac{c}{\sqrt{h}}\right)
\pm\left|\frac32-\frac{c}{\sqrt{h}}\right|\right)
\] 
one concludes convective exponential spectral stability in case (v) from the fact that when $\nu>1$
\begin{align*}
\frac32-\frac{c}{\sqrt{H_R}}&<0\,,&
\frac32-\frac{c}{\sqrt{H_L}}&>0\,.
\end{align*}

In turn, in case (i), $H_R<H_s$ and $H_L<H_s$ so that for $h=H_L$, $H_R$,
\begin{align*}
\Re(\gamma_{\pm,h}(\lambda))&>0\,,&
\textrm{ when }\Re(\lambda)\gg 1\,.
\end{align*}
Therefore to prove that convective spectral stability fails it is sufficient to prove that a spectral instability is caused by what happens near $-\infty$, thus with $h=H_L$. This follows from the fact that in case (i), $F>2$ hence both endstates generate an essential spectrum instability.

\subsection{Discontinuous fronts}

We now specialize to discontinuous fronts, as in cases (ii), (iii) and (iv) of Proposition~~\ref{existprop}. Our goal in the present section is to completely elucidate the effect of essential spectrum on stability/instability of any type so as to reduce the issues to the examination of unstable eigenvalues, carried out in the next section.

For all the cases under consideration here, $H_R<H_s<H_L$ thus
\begin{align*}
\Re(\gamma_{-,H_L}(\lambda))&>0>\Re(\gamma_{+,H_L}(\lambda))\,,&
\Re(\gamma_{\pm,H_R}(\lambda))&>0\,,&
\textrm{ when }\Re(\lambda)\gg 1\,.
\end{align*}
This readily implies that instabilities due to the behavior near $+\infty$ may always be convectively stabilized whereas the convective stabilization of instabilities due to the behavior near $-\infty$ require that those occur through a change of sign in $\Re(\gamma_{+,H_L}(\lambda))$. Note that at this stage it is not clear whether the latter necessary condition is also sufficient.

In order to decide this necessary condition, we compute that
\begin{align*}
\gamma_{\pm,h}(\lambda)
\,\stackrel{|\lambda|\to\infty}{=}\,\frac{\lambda+\frac{1}{\sqrt{h}}\left(1\mp\frac{F}{2}\right)}{c-\sqrt{h}\mp\frac{\sqrt{h}}{F}}
+\cO(|\lambda|^{-1})\,.
\end{align*}
As a consequence, 
\begin{align*}
\liminf_{\substack{|\lambda|\to\infty\\\Re(\lambda)\geq0}}\Re(\gamma_{-,H_L}(\lambda))
&=\frac{1}{H_L}
\frac{1+\frac{F}{2}}{\frac{c}{\sqrt{H_L}}-1+\frac1F}\,,&
\limsup_{\substack{|\lambda|\to\infty\\\Re(\lambda)\geq0}}\Re(\gamma_{+,H_L}(\lambda))
&=\frac{1}{H_L}
\frac{1-\frac{F}{2}}{\frac{c}{\sqrt{H_L}}-1-\frac1F}\,.
\end{align*}
In particular, the condition is at least met in the high-frequency regime. Another necessary condition is that $\Re(\lambda)>0$ implies $\Re(\gamma_{+,H_L}(\lambda))<\Re(\gamma_{-,H_L}(\lambda))$. We have already examined the latter condition when discussing absolute instability and proved that it fails only in case (ii) when 
\[
\sqrt{2\nu(\nu+1)}<F<\nu(\nu+1)\,.
\]
Moreover when $F<\sqrt{2\nu(\nu+1)}$, $\Re(\lambda)\geq0$ also implies $\Re(\gamma_{+,H_L}(\lambda))<\Re(\gamma_{-,H_L}(\lambda))$. 

The full condition we want to elucidate is 
\begin{align*}
\inf_{\Re(\lambda)\geq0}\Re(\gamma_{-,H_L}(\lambda))&>0&\textrm{and}&&
\sup_{\Re(\lambda)\geq0}\Re(\gamma_{+,H_L}(\lambda))&<\inf_{\Re(\lambda)\geq0}\Re(\gamma_{-,H_L}(\lambda))\,.
\end{align*}
With explicit expressions~\eqref{eq:gamma} in mind, we first recall that
\[
\cQ_h(\lambda)=-\Im(\lambda)^2\frac{h}{F^2}+\cQ_h(\Re(\lambda))
+\mbi\Im(\lambda)\,\left(2\Re(\lambda)\,\frac{h}{F^2}-(c-\sqrt{h})+\frac{2\sqrt{h}}{F^2}\right)
\]
and observe that in present cases, when $\Re(\lambda)\geq0$,
\begin{align*}
2\Re(\lambda)\,\frac{H_L}{F^2}-(c-\sqrt{H_L})+\frac{2\sqrt{H_L}}{F^2}
&\geq -(c-\sqrt{H_L})+\frac{2\sqrt{H_L}}{F^2}>0\,,&\\
\cQ_{H_L}(\Re(\lambda))&\geq\left(\frac32-\frac{c}{\sqrt{h}}\right)^2>0\,.&
\end{align*}

This motivates the following lemma.

\begin{lemma}\label{l:sqrt}
For any positive $\alpha$, $\gamma$,
\begin{align*}
\inf_{y\in\R}\Re(\sqrt{-y^2\,\alpha+\mbi\,y\,\beta+\gamma})
&\,=\,\min\left(\left\{\,\sqrt{\gamma}\,;\,\frac{|\beta|}{2\sqrt{\alpha}}\,\right\}\right)\,,\\
\sup_{y\in\R}\Re(\sqrt{-y^2\,\alpha+\mbi\,y\,\beta+\gamma})
&\,=\,\max\left(\left\{\,\sqrt{\gamma}\,;\,\frac{|\beta|}{2\sqrt{\alpha}}\,\right\}\right)\,.
\end{align*}
\end{lemma}

\begin{proof}
From the classical formula $(\Re(\sqrt{z}))^2=(\Re(z)+|z|)/2$, we deduce
\[
\left(\Re(\sqrt{-y^2\,\alpha+\mbi\,y\,\beta+\gamma})\right)^2
\,=\,\frac12\left(
\gamma-y^2\,\alpha+\sqrt{(\gamma-y^2\,\alpha)^2+y^2\beta^2}
\right)=:\Gamma(y^2)\,.
\]
Direct computations yield that
\begin{align*}
\Gamma(0)&=\gamma\,,&\lim_{+\infty}\Gamma =\frac{\beta^2}{4\alpha}
\end{align*}
and that either $\Gamma'$ is constantly zero, which happens when $\beta^2=4\alpha\gamma$, or that it never vanishes, when $\beta^2\neq 4\alpha\gamma$. Hence the result by monotony. 
\end{proof}

When applying the lemma to an estimate of $\Re(\cQ_h(\lambda))$, we want to determine what is the minimum obtained from the lemma. This stems from the following computation
\begin{align*}
4\,\frac{H_L}{F^2}\,&\cQ_{H_L}(\Re(\lambda))
-\left(2\Re(\lambda)\,\frac{H_L}{F^2}-(c-\sqrt{H_L})+\frac{2\sqrt{H_L}}{F^2}\right)^2\\
&=4\,\frac{H_L}{F^2}\,\left(\frac32-\frac{c}{\sqrt{H_L}}\right)^2
-\left(-(c-\sqrt{H_L})+\frac{2\sqrt{H_L}}{F^2}\right)^2\\
&=\frac{F-2}{F}\left(\frac{\sqrt{H_L}}{F}+c-\sqrt{H_L}\right)
\left(
\frac{2\sqrt{H_L}}{F}\,\left(\frac32-\frac{c}{\sqrt{H_L}}\right)
-(c-\sqrt{H_L})+\frac{2\sqrt{H_L}}{F^2}\right)\,.
\end{align*}
Note that the latter expression does not depend on $\lambda$ and that its sign is determined by the sign of $F-2$. All together we deduce that, to determine the convective stabilitization of the essential spectrum, when $F\geq 2$ it is sufficient to discuss what happens at the limit $\Im(\lambda)\to\infty$ whereas when $F<2$ it is sufficient to look at the case when $\lambda\in\R$.

As for the smooth profiles, the case $F<2$ has been thoroughly analyzed in \cite{SYZ} and the only thing left is to check that one may also obtain exponential convective stability in this case. This follows from the same computation as for smooth profiles. 

From now on we focus on discontinuous profiles when $F\geq2$. In this context it follows from the previous lemma and the above $|\lambda|\to\infty$ asymptotics that failure of 
convective stability by essential spectrum is equivalent to
\begin{align*}
\gamma_-^\infty:=\frac{1}{H_L}
\frac{1+\frac{F}{2}}{\frac{c}{\sqrt{H_L}}-1+\frac1F}
&<
\frac{1}{H_L}
\frac{\frac{F}{2}-1}{-\frac{c}{\sqrt{H_L}}+1+\frac1F}
=:\gamma_+^\infty\,.
\end{align*}
This coincides with the condition for absolute instability
\[
F>\sqrt{2\nu(\nu+1)}\,.
\]

\br
Let us emphasize that the coincidence of the boundary of absolute instability with the boundary of convective stability\footnote{For the moment we have only proved stabilization of the essential spectrum but we do prove full stability in the end.} defined through scalar exponential weights is not a general fact but a specific property of the present problem. It comes with the strong consequence that there is no need to consider more general weights. One may obtain a simple (but artificial) counterexample to a more general claim in this direction by simply considering as a single system two uncoupled systems requiring incompatible weights.
\er

\subsection{Summary} 

\begin{figure}[t!]
\begin{center}
\includegraphics[scale=0.32]{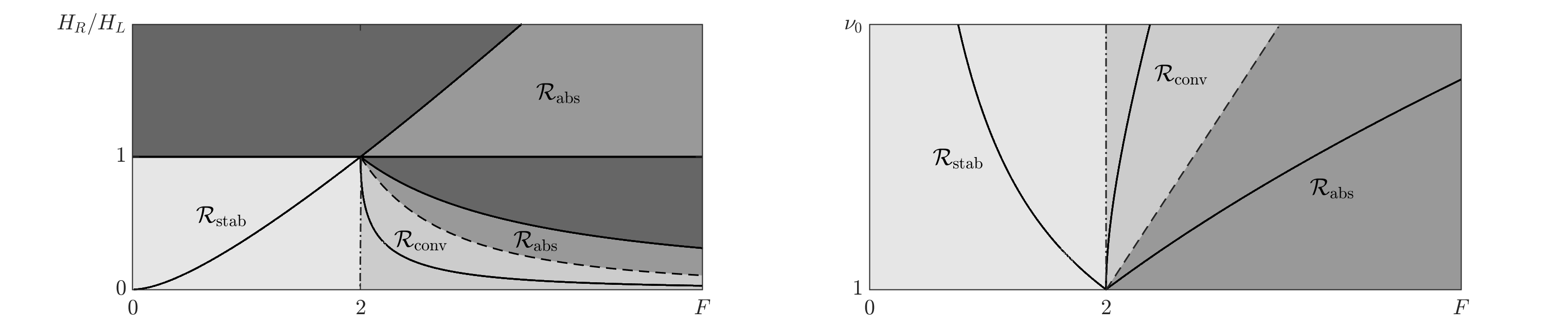}
\end{center}
\caption{Stability regimes on the domain of existence Figure \ref{existence:fig}. In region $\pazocal{R}_\mathrm{stab}$ waves are stable in the unweighted space,  while in $\pazocal{R}_\mathrm{conv}$ waves are convectively stable in the $(\eta_L,\eta_R)$-weighted. Finally, in region $\pazocal{R}_\mathrm{abs}$ waves are absolutely unstable. The boundary of the regions $\pazocal{R}_\mathrm{stab}$ and $\pazocal{R}_\mathrm{conv}$ delimited by the dash-dotted line corresponds to $F=2$ while the boundary of the regions $\pazocal{R}_\mathrm{conv}$ and $\pazocal{R}_\mathrm{abs}$ delimited by the dashed line corresponds to $F=\sqrt{2\nu(\nu+1)}$ with $\nu=\sqrt{H_L/H_R}>1$.}
\label{stability:fig}
\end{figure}

The preceding analysis motivates the definition of the following regions in parameters space. We set
\bqs
\pazocal{R}_\mathrm{stab}:=\left\{ \nu>1 \text{ and } 0<F<2 \text{ with } F\neq \frac{\nu+1}{\nu^2} \right\}, \quad \pazocal{R}_\mathrm{conv}:=\left\{ \nu>1 \text{ and } 2\leq F<\sqrt{2\nu(\nu+1)} \right\},
\eqs
together with
\bqs
\pazocal{R}_\mathrm{abs}:=\left\{ \nu>1 \text{ and } \sqrt{2\nu(\nu+1)}<F<\nu(\nu+1)\right\}\cup \left\{ 0<\nu<1 \text{ and } \frac{\nu+1}{\nu^2}<F\right\}.
\eqs
We refer to Figure~\ref{stability:fig} for a visualization of these regions in parameters space.  So far, our results can be summarized as:
\begin{itemize}
\item In region $\pazocal{R}_\mathrm{stab}$, waves of cases (iv)-(v) have marginally stable essential spectrum and convective exponential stabilization of the essential spectrum can always be achieved.  
\item In region $\pazocal{R}_\mathrm{conv}$, waves of cases (ii)-(iii)-(iv) have unstable essential spectrum but they have convectively exponentially stable essential spectrum in some $(\eta_L,\eta_R)$-weighted spaces with $\eta_L\geq 0$ and $\eta_R\leq 0$.
\item In region $\pazocal{R}_\mathrm{abs}$, waves of cases (i)-(ii) have unstable essential spectrum with absolute instability in the sense that the essential spectrum can not be stabilized in any $(\eta_L,\eta_R)$-weighted spaces with, in case (ii), the presence of unstable branch points for the resolvent operator. 
\end{itemize}
Using the results of \cite{SYZ}, in region $\pazocal{R}_\mathrm{stab}$, waves of cases (iv)-(v) are marginally spectrally stable in the sense that the spectrum is included in $\left\{\lambda\in\C;\Re(\lambda)<0\right\}\cup\left\{0\right\}$ with an embedded eigenvalue at $\lambda=0$, of multiplicity one in a generalized sense. As a consequence, there remains to study whether when $2\leq F<\sqrt{2\nu(\nu+1)}$ and $\nu>1$, that is in region $\pazocal{R}_\mathrm{conv}$, there is a choice of $\eta_L\in (\gamma_+^\infty,\gamma_-^\infty)$ such that when $\eta_R$ is sufficiently negative there is no $\lambda$ with $\Re(\lambda)\geq0$ possessing an eigenfunction in a $(\eta_L,\eta_R)$-weighted space. This is the object of the next section. As a preliminary we observe that by taking $\eta_R$ sufficiently negative we may readily discard eigenfunctions that are not zero on $\R_+$.

\subsection{Maximal decay rate: another view}

Before moving on with the rest of the program, we would like to halt and offer a different perspective on the former computations so as to address the following question: what is the maximal essential spectral gap that may be opened by tuning our weights appropriately ? Since the boundaries of the essential spectrum due to what happens near $+\infty$ may be pushed arbitrarily to the left of the compex plane, we may again focus on the contribution from the left. The same computation we have carried out to determine absolute instability yields as an upper bound for the essential spectral gap
\[
\theta_{opt}\,:=\,
\frac{F^2}{2H_L}\left(-(c-\sqrt{H_L})+\frac{2\sqrt{H_L}}{F^2}\right)\,,
\]
and that it is reached with spatial decay rate
\[
\eta_{opt}:=\frac{1}{\left(c-\sqrt{H_L}\right)^2-\frac{H_L}{F^2}}
\,\left(-\theta_{opt}\,(c-\sqrt{H_L})-\left(\frac32-\frac{c}{\sqrt{H_L}}\right)\right)
=\frac{F^2}{2\,H_L}\,.
\]
Reciprocally one may check with arguments similar to the ones used above (mostly relying on Lemma~\ref{l:sqrt}) that choosing $\eta_L=\eta_{opt}$ and $\eta_R$ sufficiently negative provides the optimal spectral gap. 

The effect of moving $\eta_L$ is illustrated in Figure~\ref{fig:sepctrumL}. There the curves are obtained by solving in $\lambda\in\C$ the equations
\[
\eta_L+\mbi\xi=\gamma_{\pm,H_L}(\lambda)
\]
with parameter $\xi\in\R$ as
\[
\lambda=(\eta_L+\mbi\xi)\left(c-\sqrt{H_L}\right)-\frac{1}{\sqrt{H_L}}
\pm\sqrt{(\eta_L+\mbi\xi)^2\frac{H_L}{F^2}-(\eta_L+\mbi\xi)+\frac{1}{H_L}}
\]
and we have introduced 
\begin{align*}
\eta_L^\mathrm{min}&:=\gamma_+^\infty=\frac{1}{H_L}
\frac{\frac{F}{2}-1}{-\frac{c}{\sqrt{H_L}}+1+\frac1F}\,,&
\eta_L^\mathrm{max}&:=\gamma_-^\infty=\frac{1}{H_L}
\frac{1+\frac{F}{2}}{\frac{c}{\sqrt{H_L}}-1+\frac1F}\,.
\end{align*}

\br\label{follow-up-NL}
To prove the last claim in Remark~\ref{nonlinearrmk}, we observe that 
\[
\eta_L^\infty-\eta_L^\mathrm{max}
\,=\,\frac{(F-2)\,\left(\frac{c}{\sqrt{H_L}}-1+\frac1F\right)}{2\left(\frac{H_L}{F^2}-\left(c-\sqrt{H_L}\right)^2\right)}
\]
is indeed positive in the cases under consideration. To carry out the above computation, we have used that 
\[
\left(H_L - H_R\right)\left(H_L-H_{out}\right)
\,=\,3H_L^2-2c\,(c\,H_L-q_0)
\,=\,3H_L^2-2c\,H_L^{\frac32}\,.
\] 
\er

\begin{figure}[!t]
\subfigure[No weight ($\eta_L=0$).]{\includegraphics[scale=0.32]{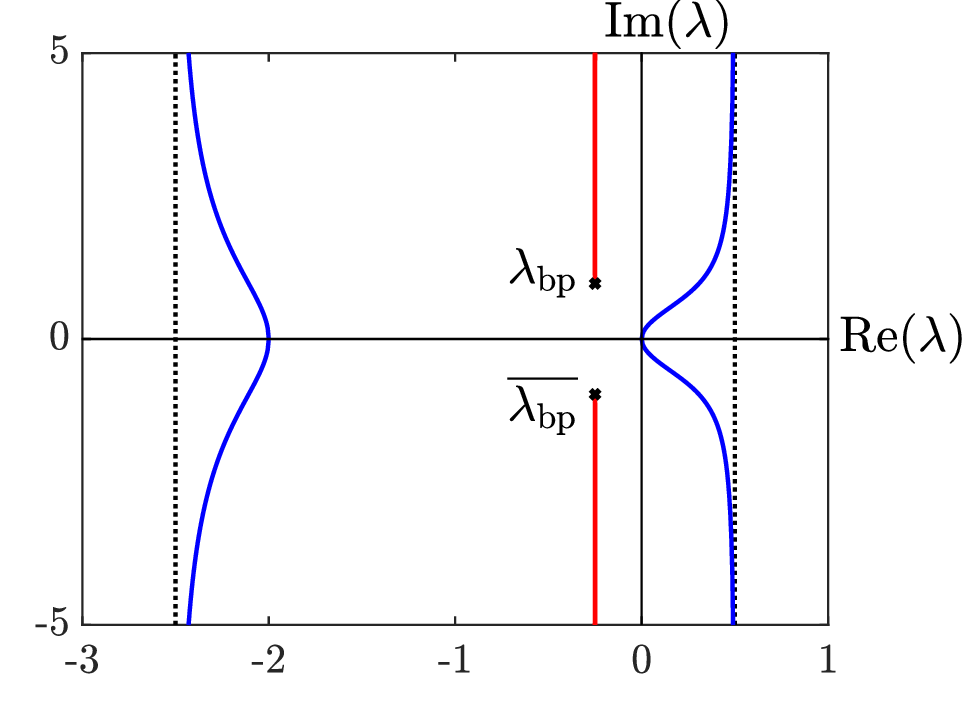}}
\subfigure[$\eta_L\in\left(0,\eta_L^\mathrm{min}\right)$.]{\includegraphics[scale=0.32]{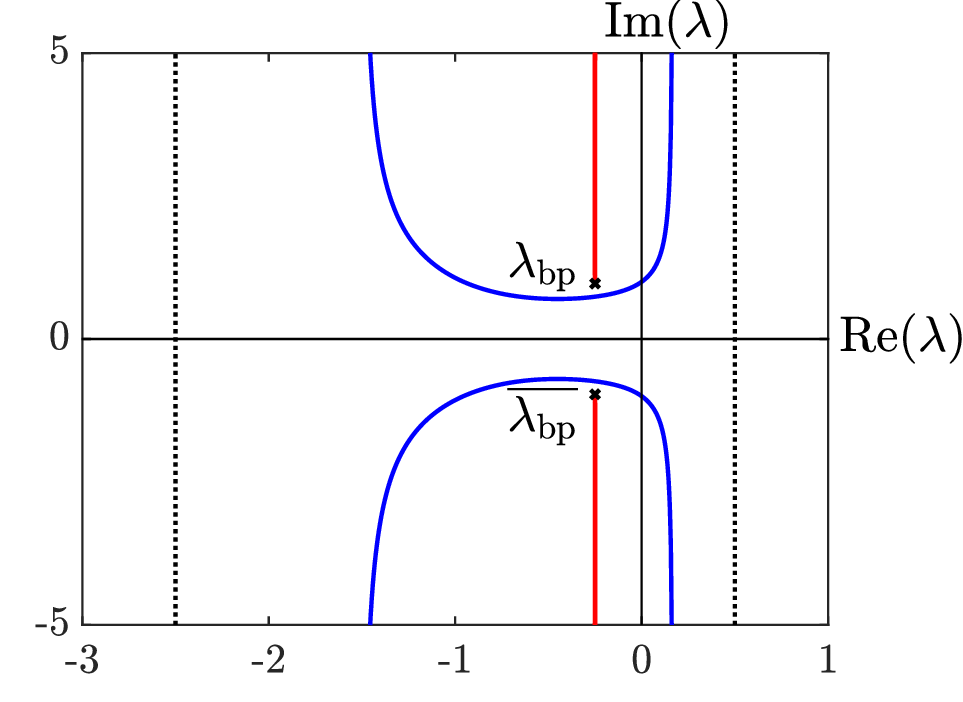}}
\subfigure[$\eta_L=\eta_L^\mathrm{min}$.]{\includegraphics[scale=0.32]{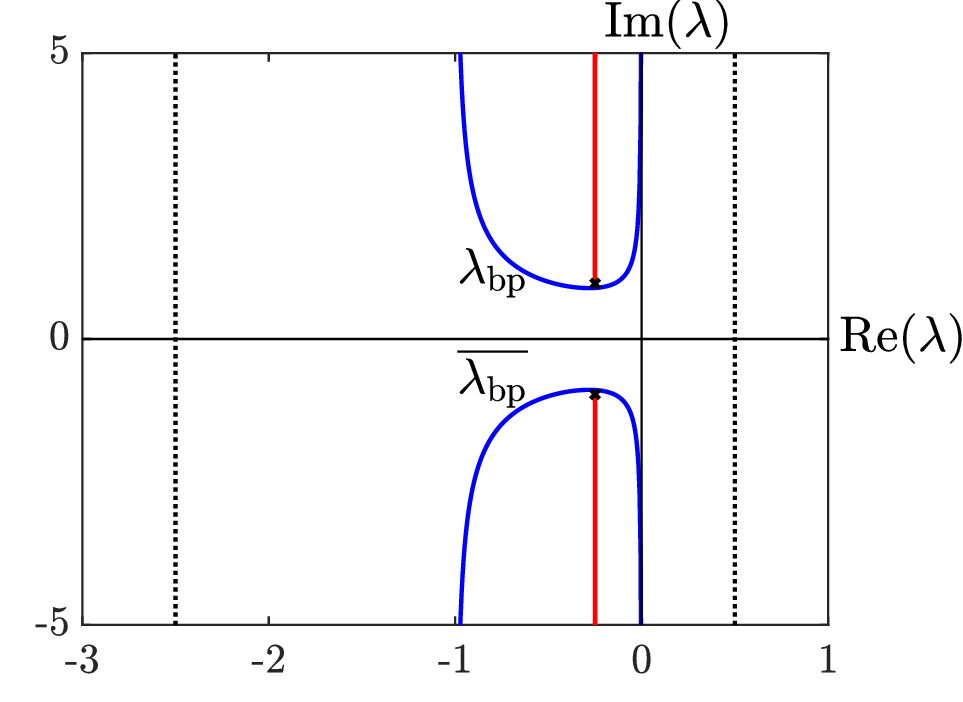}}
\subfigure[$\eta_L\in\left(\eta_L^\mathrm{min},\eta_L^\mathrm{max}\right)$.]{\includegraphics[scale=0.32]{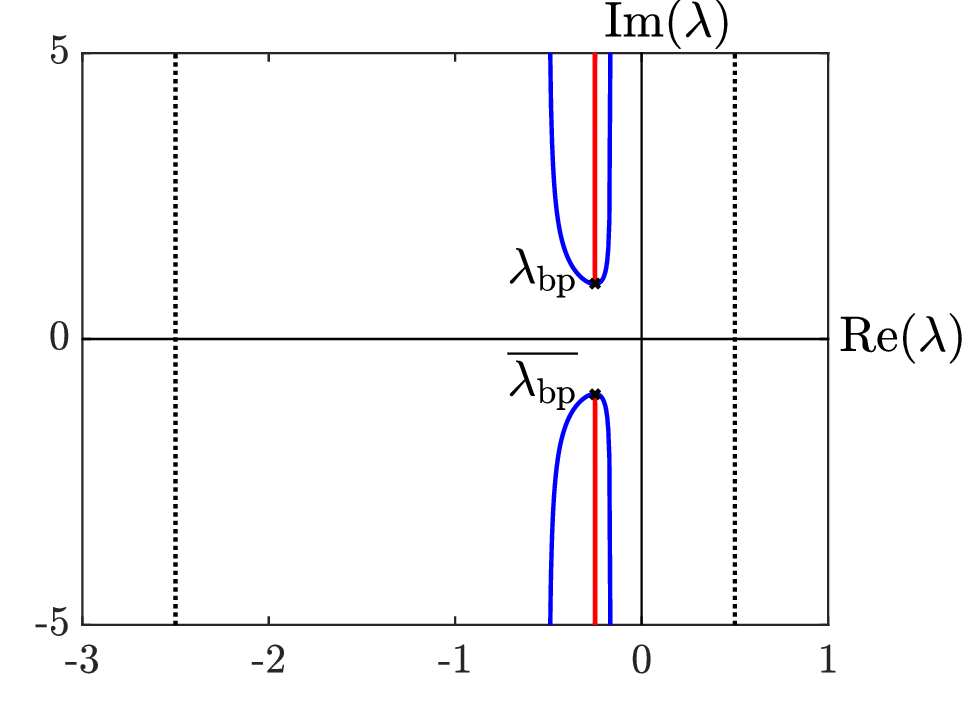}}
\subfigure[$\eta_L=\eta_L^\mathrm{max}$.]{\includegraphics[scale=0.32]{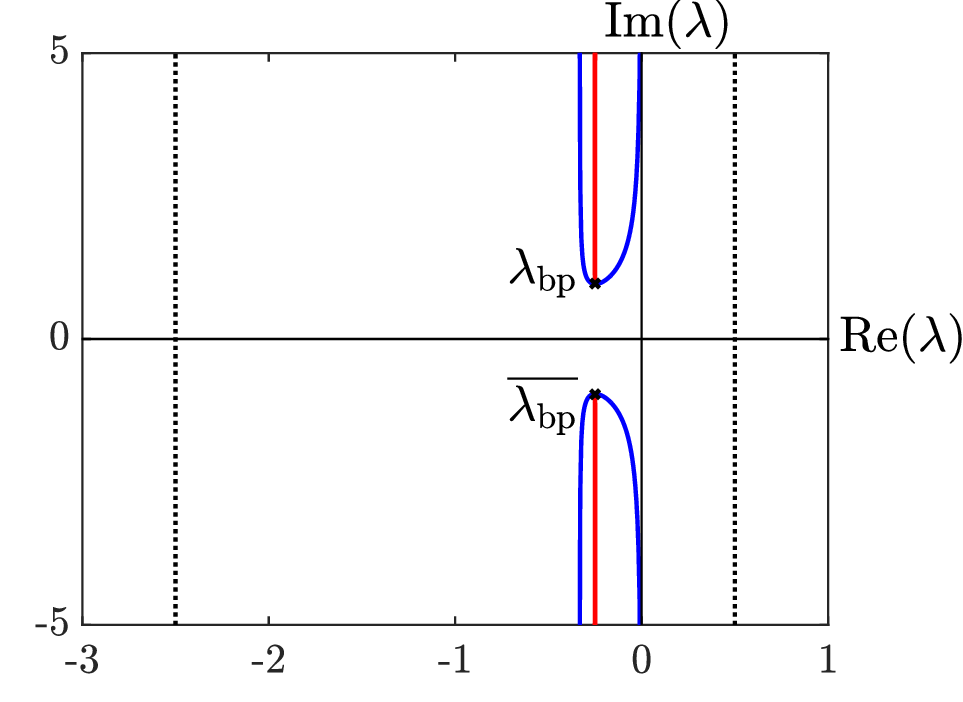}}
\subfigure[$\eta_L>\eta_L^\mathrm{max}$.]{\includegraphics[scale=0.32]{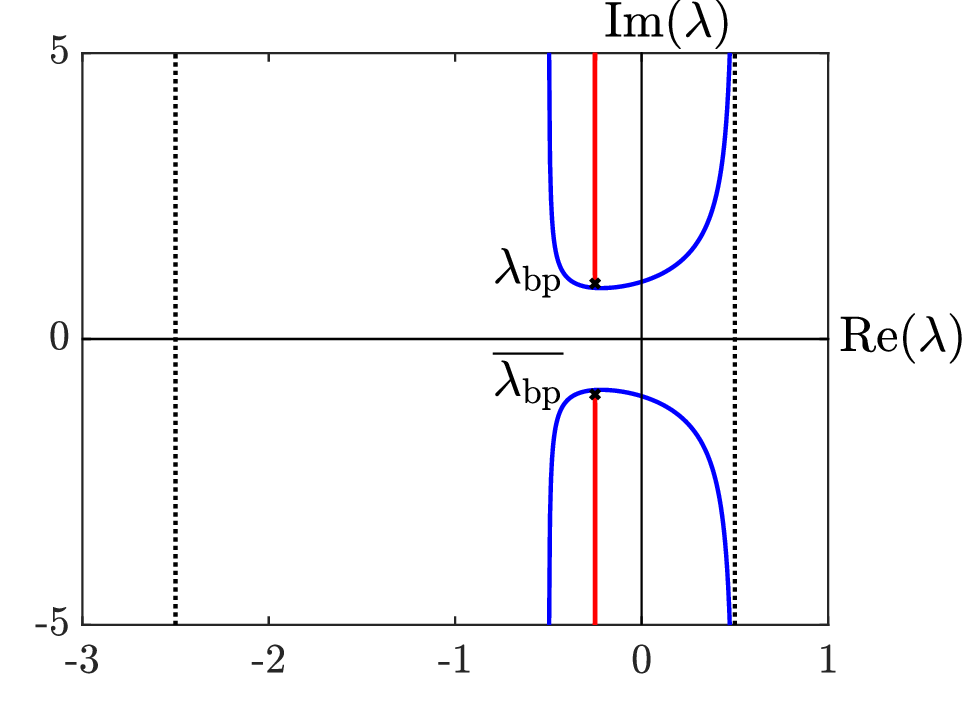}}
\caption{Visualization of some of the boundaries of the weighted essential spectrum as $\eta_L$ is varied. In all figures, the blue curves represent the computed boundaries while the red half-lines 
represent the absolute spectrum which terminates at branch points marked by black crosses. Without weight, the essential spectrum is always unstable, while it is strictly stabilized with a gap for $\eta_L\in\left(\eta_L^\mathrm{min},\eta_L^\mathrm{max}\right)$. At the critical weights $\eta_L=\eta_L^\mathrm{min}$ or $\eta_L=\eta_L^\mathrm{max}$, the weighted essential spectrum is only marginally stabilized with no spectral gap. The parameters are fixed to $(H_L,H_R,F)=(1,1/4,3)$ so that $(\eta_L^\mathrm{min},\eta_L^\mathrm{max})=(3,5)$.}
\label{fig:sepctrumL}
\end{figure}

\section{Sturm--Liouville analysis}\label{s:sl}

Throughout this section, we consider a discontinuous profile and assume that $F$ satisfies
\[
2\leq F<\sqrt{2\nu(\nu+1)}\, \text{ with } \nu>1.
\]
We also fix some $\eta_L\in\left(\eta_L^\mathrm{min},\eta_L^\mathrm{max}\right)$ where $\eta_L^\mathrm{min}$ and $\eta_L^\mathrm{max}$ are as above.

Here, as announced, we study possible unstable eigenvalues and, to do so, adapt the arguments from \cite{SYZ}.

\subsection{The reduced eigenvalue problem}

By imposing a vanishing on $\R_+$, as we can for our purposes, we reduce the eigenfuction problem to finding a nonzero $(v,\psi)$ in\footnote{With obvious notational adaptation for weighted spaces.} $H^1_{\eta_L}(\R_-^*;\C^2)\times \C$ solving
\[
\left\{
\begin{array}{rl}
\lambda v+(A\,v)'&=\ E\,v,\qquad\textrm{on }\R_-^*,\\
\psi\left[\,\lambda\,W-r(W)\right]
&=\ \left[Av\right]\,.
\end{array}
\right.
\]

We begin by inspecting the special case when $v$ is zero (but $\psi$ is not). A direct inspection shows that it only happens when $\left[\,\lambda\,W-r(W)\right]$ is zero, which is equivalent to $\lambda$ and $[r(W)]$ both being zero (since the first component of $[r(W)]$ is zero and the first of $[W]$ is nonzero). The latter occurs exactly when we are in the Riemann shock case, case (iii) of Proposition~\ref{existprop}. Actually the vanishing of $\left[\,\lambda\,W-r(W)\right]$ when $\lambda=0$ alter many of the considerations to come. For this reason we postpone the treatment of the Riemann shock case to the end of the present section. 

Since we are now excluding the Riemann shock case, $\left[\,\lambda\,W-r(W)\right]$ is non zero and one may eliminate $\psi$ to reduce the discussion further to the existence of a nonzero $v=(v_1,v_2)$ in $H^1_{\eta_L}(\R_-^*;\C^2)$ such that on $\R_-^*$
\[
\left\{
\begin{array}{l}
\lambda\,v_1+(-cv_1+v_2)'=0,\\
\lambda\,v_2+\left(\left(-\frac{Q^2}{H^2}+\frac{H}{F^2}\right)\,v_1+\left(-c+2\frac{Q}{H}\right) \,v_2\right)'=\ \left(1+2\frac{Q^2}{H^3}\right)v_1-2\frac{Q}{H^2}\,v_2\,,
\end{array}
\right.
\]
and 
\[
(-cv_1+v_2)(0^-)\times\left[\lambda Q-\left(H-\frac{Q^2}{H^2}\right)\right]
-\left(\left(-\frac{Q^2}{H^2}+\frac{H}{F^2}\right)\,v_1+\left(-c+2\frac{Q}{H}\right) \,v_2\right)(0^-)\times \left[\lambda H\right]\,=\,0\,.
\]

For the sake of writing simplification we introduce one flux coordinate and replace $v$ with $u=(u_1,u_2):=(v_1,-cv_1+v_2)$. With this change, we turn the problem into finding a nonzero $u$ in $H^1_{\eta_L}(\R_-^*;\C^2)$ such that on $\R_-^*$
\[
\left\{
\begin{array}{l}
\lambda\,u_1+u_2'=0,\\
(au_1)'=\left(1-2\left(c-\frac{Q}{H}\right)\frac{Q}{H^2}-2\lambda\,\left(c-\frac{Q}{H}\right)\right)u_1-\left(\lambda+2\frac{Q}{H^2}+2\left(\frac{Q}{H}\right)'\right)\,u_2\,,
\end{array}
\right.
\]
and 
\[
u_2(0^-)\times\left[\lambda Q-\left(H-\frac{Q^2}{H^2}\right)\right]
-\left(a\,u_1+\left(-c+2\frac{Q}{H}\right) \,u_2\right)(0^-)\times \left[\lambda H\right]\,=\,0\,.
\]
In the foregoing we have denoted by $a$ the characteristic determinant
\[
a:=
\frac{H}{F^2}-\left(c-\frac{Q}{H}\right)^2\,.
\]

Let us now examine the possibility to have a nonzero solution $u$ with zero component $u_2$. A direct inspection shows that this may happen only when $\lambda=0$ and that the corresponding $u$ is necessarily a multiple of $(H',0)$. Note that reciprocally one checks readily that when $\lambda=0$ necessarily $u_2\equiv0$. Thus this situation corresponds exactly to the possibility of $0$ being in the spectrum due to translational invariance.

We now focus on the case when $u_2$ is not zero. Then the eigenvalue problem may be recasted into the problem of finding a nonzero $u_2$ in $H^2_{\eta_L}(\R_-^*;\C)$ solving
\[
(a\,u_2')'=\left(1-2\left(c-\frac{Q}{H}\right)\frac{Q}{H^2}-2\lambda\,\left(c-\frac{Q}{H}\right)\right)u_2'+\left(\lambda+2\frac{Q}{H^2}+2\left(\frac{Q}{H}\right)'\right)\,\lambda\,u_2\,,
\]
on $\R_-^*$ and 
\[
u_2(0^-)\times\left[\lambda Q-\left(H-\frac{Q^2}{H^2}\right)\right]
-\left(-a\,u_2'+\left(-c+2\frac{Q}{H}\right)\,\lambda\,u_2\right)(0^-)\times \left[ H\right]\,=\,0\,.
\]
In order to match notation from \cite{SYZ} we introduce
\begin{align*}
f_1&:=\frac{2}{a}\left(c-\frac{Q}{H}\right),&
f_2&:=-\frac{1}{a}\left(1-2\left(c-\frac{Q}{H}\right)\frac{Q}{H^2}-a'\right),\\
f_3&:=-\frac{1}{a},&f_4&:=-\frac{2}{a}
\left(\frac{Q}{H^2}+\left(\frac{Q}{H}\right)'\right)\,,
\end{align*}
so that the equation on $\R_-^*$ becomes
\begin{align*}
u_2''+\left(f_1\lambda+f_2\right)u_2'+\lambda\,\left(f_3\lambda+f_4\right)\,u_2\,=\,0\,.
\end{align*}
We point out for later use that from the fact that $(H',0)$ solves the interior ODE problem for $(u_1,u_2)$ when $\lambda=0$, one deduces that $H''=-f_2\,H'$.

At last, in order to symmetrize the interior equation, we perform a Liouville-type transformation and replace $u_2$ with $w$ defined by\footnote{It should not be confused with the $w$ used in the initial introduction of the spectral problem.}
\[
w(x):=\exp\left(\frac{1}{2}\int_0^{x}(f_1\lambda+f_2)\right)u_2(x)\,.
\]
This replaces the equation on $\R_-^*$ with
\[
w''+\left(\lambda\,\left(f_3\lambda+f_4\right)
-\frac14\left(f_1\lambda+f_2\right)^2-\frac12\left(f_1\lambda+f_2\right)'\right)w=0,
\]
also written as
\be 
\label{weq}
w''+\left(\left(f_3-\frac{1}{4}f_1^2\right)\lambda^2+\left(f_4-\frac{1}{2}f_1f_2-\frac{1}{2}f_1'\right)\lambda-\frac{1}{4}f_2^2-\frac{1}{2}f_2'\right)w=0,
\ee 
which is exactly \cite[Equation~(2.14)]{SYZ}, whereas the boundary condition becomes 
\be
w'(0^-) =(c_1\lambda+c_2)\,w(0^-)\,,
\label{BCw0}
\ee
where
\begin{align*}
c_1&:=\frac12f_1(0^-)-\frac{[Q]}{a(0^-)[H]}
+\frac{1}{a(0^-)}\left(-c+2\frac{Q(0^-)}{H(0^-)}\right)
=\frac{1}{a(0^-)}\left(-c+\frac{Q(0^-)}{H(0^-)}\right)
=-\frac{1}{a(0^-)}\frac{q_0}{H(0^-)}\\
c_2&:=\frac12f_2(0^-)+\frac{\left[H-\frac{Q^2}{H^2}\right]}{a(0^-)[H]}\,.
\end{align*}

Before going on we need to check that $u_2\in H^2_{\eta_L}(\R_-^*;\C)$ implies $w\in H^2(\R_-^*;\C)$. From the analysis of the previous section we know that, when $u_2\in H^2_{\eta_L}(\R_-^*;\C)$, its spatial decay rate is precisely $\Re(\gamma_{-,H_L}(\lambda))$. Therefore this amounts to proving that
\[
\Re(\gamma_{-,H_L}(\lambda))>-\frac12\left(\Re(\lambda)\lim_{-\infty}(f_1)+\lim_{-\infty}(f_2)\right)\,.
\]
A direct computation shows that this is equivalent to
\[
\Re\left(\sqrt{\cQ_{H_L}(\lambda)}\right)>0\,,
\]
thus to the fact that $\lambda$ does not belong to the absolute spectrum.

Therefore it is indeed sufficient to discard the possibility of a nonzero $w$ in $H^2(\R_-^*;\C)$ solving \eqref{weq}--\eqref{BCw0}. 

\subsection{Non real growth rates}\label{s:nonreal}

We stress that whereas the interior part, \eqref{weq}, is symmetric on functions compactly supported in $\R_-^*$, completing it with boundary condition \eqref{BCw0} does not yield a symmetric operator. An argument, specialized to the case at hand, is thus needed to show that necessarily $\lambda\in\R$ if such a $w$ exists. We provide such a concrete argument now.

To begin with, we observe that, combined with \eqref{BCw0}, multiplying \eqref{weq} with $\overline{w}$ and integrating yield
\be
\begin{split}
(c_1\lambda&+c_2)\left|w(0^-)\right|^2-\int_{\R_-^*}\left|w'\right|^2\\
&+ \int_{\R_-^*} |w|^2\left(\left(f_3-\frac{1}{4}f_1^2\right)\lambda^2+\left(f_4-\frac{1}{2}f_1f_2-\frac{1}{2}f_1'\right)\lambda-\frac{1}{4}f_2^2-\frac{1}{2}f_2'\right)\,=\,0.
\end{split}
\label{eigvalpbm}
\ee
When $\Im(\lambda)\neq0$, the imaginary part of \eqref{eigvalpbm} gives
\[
c_1\left|w(0^-)\right|^2
+\,\int_{\R_-^*} |w|^2\left(\left(f_3-\frac{1}{4}f_1^2\right)2\Re(\lambda)+\left(f_4-\frac{1}{2}f_1f_2-\frac{1}{2}f_1'\right)\right)\,=\,0.
\]

Since $c_1<0$ and $f_3<0$, the last equality implies
\[
\Re(\lambda)<-\frac{\inf_{\R_-^*}\left(-f_4+\frac{1}{2}f_1f_2+\frac{1}{2}f_1'\right)}{\inf_{\R_-^*}\left(|f_3|+\frac{1}{4}f_1^2\right)}\,.
\]

As a consequence, we need to study the sign of $f_4-\frac{1}{2}f_1f_2-\frac{1}{2}f_1'$. To determine this sign we observe that
\begin{align*}
a^2\,H^4\,\left(f_4-\frac{1}{2}f_1f_2-\frac{1}{2}f_1'\right)
&=-\frac{2c}{F^2}H^4+\frac{2q_0^2}{F^2}H^3+c^2q_0H^2-2cq_0^2H+q_0^3
=-\frac{2Q}{F^2} \cZ(H)\,.
\end{align*}
with 
\[
\cZ(h):=h^3-\frac{1}{2}F^2cq_0h+\frac12F^2q_0^2\,.
\]
Let us denote $H_c$ the positive root of $\cZ'$, that is,
\[
H_c:=\frac{F\,\sqrt{c\,q_0}}{\sqrt{6}}
=\frac{F\,H_R\,\nu\,\sqrt{\nu^2+\nu+1}}{\sqrt{6}(\nu+1)}\,.
\]
On $[H_c,+\infty)$, $\cZ$ is increasing.

We directly borrow from \cite[Section~4.1]{SYZ} that when $F>(\nu+1)/\nu^2$, one has $H_*>H_c$ and, when moreover $H_*\leq H_L$, $\cZ(H_*)>0$. This directly implies that in cases (iii) and (iv) of Proposition~\ref{existprop}, indeed $\inf_{\R_-^*}\left(-f_4+\frac{1}{2}f_1f_2+\frac{1}{2}f_1'\right)>0$.

To complete the analysis of non real eigenvalues, we only need to show that in case (ii), $H_L>H_c$ and $\cZ(H_L)>0$. It is straightforward to check that when $F\leq \sqrt{2\nu(\nu+1)}$ indeed $H_L>H_c$, whereas $\cZ(H_L)>0$ is exactly equivalent to $F< \sqrt{2\nu(\nu+1)}$.

This achieves the proof that a spectral gap is present for non real eigenvalues.

\subsection{Real growth rates}\label{s:real}

We now turn our attention to the case of real eigenvalues. Throughout the present subsection, we assume that $\lambda\in\R_+$ and our goal is again to rule out the possibility of a nonzero $w$ in $H^2(\R_-^*;\C)$ solving \eqref{weq}--\eqref{BCw0}.

Our starting point is again Equation \eqref{eigvalpbm}, that we write now as
\[
\cB(w,w)\,=\,-\lambda\,\cA_\lambda(w,w)
\]
where $\cA_\lambda$ and $\cB$ are the symmetric sesquilinear\footnote{Consistently with our convention for scalar products, they are linear in their second factors.} forms on $H^1(\R_-^*;\C)$ defined through their quadratic forms
\begin{align*}
\cA_\lambda(v,v)&:=-c_1\left|v(0^-)\right|^2
+\,\int_{\R_-^*} |v|^2\left(\left(-f_3+\frac{1}{4}f_1^2\right)\lambda+\left(-f_4+\frac{1}{2}f_1f_2+\frac{1}{2}f_1'\right)\right)\,,\\
\cB(v,v)&:=-c_2\left|v(0^-)\right|^2
+\,\int_{\R_-^*} |v'|^2
+\,\int_{\R_-^*} |v|^2\left(\frac{1}{4}f_2^2+\frac{1}{2}f_2'\right)\,.
\end{align*}
Note that since $\lambda\in\R_+$, the analysis of the former subsection yields that $\cA_\lambda$ is positive definite when $F< \sqrt{2\nu(\nu+1)}$. In order to conclude it is therefore sufficient to prove that $\cB$ is also positive definite, that is, 
\begin{align*}
&0<\inf_{\substack{v\in H^1\\v\not\equiv0}}\frac{\cB(v,v)}{\|v\|_{H^1}^2}\,,&
\textrm{or equivalently}&&
&0<\inf_{\substack{v\in H^1\\v\not\equiv0}}\frac{\cB(v,v)}{\|v\|_{L^2}^2}\,.
\end{align*}
The equivalence between the two conditions follows from the following G\aa rding-type inequality: there exist positive $c$ and $C$ such that, for any $v\in H^1$, $\cB(v,v)\geq c\|v\|_{H^1}^2-C\|v\|_{L^2}^2$. A refined version of the latter is proved below.

As in \cite{SYZ}, we prove the latter by a continuity/homotopy argument. To set it, we introduce, for $x_0\in\R_-$, $\cB_{x_0}$ the symmetric sesquilinear form on $H^1(\R_-^*;\C)$ defined through its quadratic form
\begin{align*}
\cB_{x_0}(v,v)&:=-c_2^{[x_0]}\left|v(0^-)\right|^2
+\,\int_{\R_-^*} |v'|^2
+\,\int_{\R_-^*} |v|^2\left(\frac{1}{4}f_2^2+\frac{1}{2}f_2'\right)(\cdot+x_0)\,.
\end{align*}
The explicit definition of $c_2^{[x_0]}$ is given below but let us already anticipate that our choice ensures that $c_2^{[x_0]}$ depends smoothly on $x_0$ and converges as $x_0\to-\infty$ to a negative value. Note moreover that 
\[
\lim_{-\infty} f_2=-\lim_{-\infty} \frac{H''}{H'}<0\,.
\]
This implies that when $x_0$ is sufficiently close to $-\infty$, $\cB_{x_0}$ is positive definite. To motivate the expression for $c_2^{[x_0]}$, we first observe that 
\[
c_2\,=\,\frac12\,f_2(0^-)+\frac{(H(0^-)-H_L)(H(0^-)-H_{out})}{a(0^-)\,H(0^-)^2}\,.
\]
then, consistently we set
\[
c_2^{[x_0]}\,:=\,\frac12\,f_2(x_0)+\frac{(H(x_0)-H_L)(H(x_0)-H_{out})}{a(x_0)\,H(x_0)^2}\,.
\]
Note that as announced
\[
\lim_{x_0\to-\infty}c_2^{[x_0]}\,=\,\frac12\lim_{-\infty} f_2<0\,.
\]

The continuity argument is applied to the continuous function
\begin{align*}
\R_-\to \R\,,\qquad x_0\mapsto 
\inf_{\substack{v\in H^1\\v\not\equiv0}}\frac{\cB_{x_0}(v,v)}{\|v\|_{L^2}^2}\,.
\end{align*}
The fact that the foregoing function is indeed defined follows again from the G\aa rding inequality mentioned above. To complete our study of cases (ii) and (iv), it is sufficient to prove that this function cannot vanish. This follows in a straightforward way from the series of two lemmas stated and proved below.

To prepare the lemmas, we first quantify the possible failure of coercivity. For any
\[
0<\kappa<\frac14(\lim_{-\infty} f_2)^2
\]
there exist positive $\eta_\kappa$, $c_\kappa$ and $C_\kappa$ such that for any $x_0\in\R_-$ and $v\in H^1$,
\[
\cB_{x_0}(v,v)+C_\kappa\,\|v\|_{L^2((-\eta_\kappa,0))}^2
\geq c_\kappa\,\|v\|_{H^1(\R_-^*)}^2
+\kappa\,\|v\|_{L^2((-\infty,-\eta_\kappa))}^2\,.
\]
Indeed $\eta_\kappa$ may be chosen by imposing
\[
\inf_{(-\infty,-\eta_\kappa)}\left(\frac{1}{4}f_2^2+\frac{1}{2}f_2'\right)
>\frac12\left(\kappa+\frac14(\lim_{-\infty} f_2)^2\right)
\]
and the existence of $c_\kappa$ and $C_\kappa$ is a consequence of rough bounds on coefficients and the following Sobolev inequality,
\be\label{eq:Sobolev}
|v(0^-)|^2\leq \|v\|_{L^\infty((-\eta,0))}^2
\lesssim \|v\|_{L^2((-\eta,0))}\,\|v\|_{H^1((-\eta,0))}\,,
\ee
that holds for any $\eta>0$ (with an implicit constant depending on $\eta$).

As a second and last preliminary to lemmas, we find it convenient to explicitly introduce the self-adjoint operator on $L^2(\R_-^*;\C)$, $\cL_{x_0}$, of domain denoted $D_{x_0}$, associated with $\cB_{x_0}$. Explicitly
\begin{align*}
D_{x_0}&\,:=\,
\left\{\,v\in H^1(\R_-^*;\C)\,|\,\cB_{x_0}(v,\cdot)
\quad\textrm{is continuous on }L^2(\R_-^*;\C)\ \,\right\}\\
&\ \,=\,\left\{\,v\in H^2(\R_-^*;\C)\,|\,v'(0^-)=c_2^{[x_0]}\,v(0^-)\,\right\}
\end{align*}
and for $v\in D_{x_0}$,
\[
\cL_{x_0} v\,=\,-v''+v\,\left(\frac{1}{4}f_2^2+\frac{1}{2}f_2'\right)(\cdot+x_0)\,.
\]

\begin{lemma}
If
\[
0=\inf_{\substack{v\in H^1\\v\not\equiv0}}\frac{\cB_{x_0}(v,v)}{\|v\|_{L^2}^2}
\]
then there exists $v\in D_{x_0}$, $v\not\equiv0$, such that $\cL_{x_0}v=0$.
\end{lemma}

\begin{proof}
Let us consider $(v_k)_{k\in\NN}$ a minimizing sequence, normalized by $\|v_k\|_{L^2}=1$. From the G\aa rding estimate, we know that $(v_k)_{k\in\NN}$ is bounded in $H^1$ and thus, up to extracting a subsequence, we may assume that $(v_k)_{k\in\NN}$ converges weakly in $H^1$ to some $v_\infty\in H^1$. As a direct consequence of the Hahn-Banach theorem, we deduce that 
\begin{align*}
\int_{\R_-^*} |v_\infty'|^2
&+\,\int_{\R_-^*} |v_\infty|^2\left(\frac{1}{4}f_2^2+\frac{1}{2}f_2'\right)_+(\cdot+x_0)\\
&\leq \liminf_{k\to\infty}\left(
\int_{\R_-^*} |v_k'|^2
+\,\int_{\R_-^*} |v_k|^2\left(\frac{1}{4}f_2^2+\frac{1}{2}f_2'\right)_+(\cdot+x_0)\right)\,.
\end{align*}
Now pick $\eta>0$ such that $f_2^2+2f_2'$ is positive outside $(-\eta,0)$ and note that since $H^1((-\eta,0))$ is compactly embedded in $L^2((-\eta,0))$ we may assume that $(v_k)_k$ converges strongly to $v_{\infty}$ in $L^2((-\eta,0))$. Combined with \eqref{eq:Sobolev}, this is sufficient to take the limit $k\to\infty$ in the remaining part of $\cB_{x_0}(v_k,v_k)$. As a result
\[
\cB_{x_0}(v_\infty,v_\infty)\leq \lim_{k\to\infty}\cB_{x_0}(v_k,v_k)=0\,.
\]

We now prove that $v_\infty$ is nonzero. This is the place where we use the refined version of the G\aa rding estimate. Indeed it implies that there exist positive $\eta'$ and $K$ such that when $k$ is sufficiently large so as to force that $\cB_{x_0}(v_k,v_k)$ is sufficiently small
\[
\|v_k\|_{L^2((-\eta',0))}^2
\geq K\,\|v_k\|_{L^2((-\infty,-\eta'))}^2\,.
\] 
Since $\|v_k\|_{L^2}=1$, we deduce that 
\[
0<\liminf_{k\to\infty}\|v_k\|_{L^2((-\eta',0))}\,.
\]
Then, since we may assume that $(v_k)_k$ converges strongly to $v_{\infty}$ in $L^2((-\eta',0))$, we conclude that $v_\infty$ is nonzero.

Let us set $v:=v_\infty/\|v_\infty\|_{L^2}$. The vector $v$ is nonzero and satisfies $\cB_{x_0}(v,v)\leq0$, thus $\cB_{x_0}(v,v)=0$. Since $v$ minimizes the quadratic form associated with $\cB_{x_0}$ among vectors of $H^1$ with unit $L^2$ norm, there exists $\mu\in\C$ such that $\cB_{x_0}(v,\cdot)\,=\,\mu\,\langle v,\cdot\rangle_{L^2}$. In particular $v\in D_{x_0}$ and $\cL_{x_0}v=\overline{\mu} v$. Since $v$ is nonzero, evaluating the relation at $v$ shows that $\mu=0$ and concludes the proof of the lemma.
\end{proof}

The foregoing lemma is very close to many standard results but, unfortunately, we have not found a directly applicable version in the literature. Hence the above proof.

\begin{lemma}
If $v\in D_{x_0}$ is such that $\cL_{x_0}v=0$ then $v\equiv0$.
\end{lemma}

\begin{proof}
Note that the set of $v\in H^2(\R_-)$ such that  
\[
-v''+v\,\left(\frac{1}{4}f_2^2+\frac{1}{2}f_2'\right)(\cdot+x_0)\,=\,0
\]
is one-dimensional. Moreover from the fact that $(H',Q')$ solves the interior spectral ODE system in original formulation, we deduce that $(H-H_L)''=-f_2\,(H-H_L)'$, and thus that 
\begin{align*}
v^{[x_0]}\,:&\,\R_-\to\R,\,&x&\mapsto e^{-\frac12\int_{x+x_0}^{x_0} f_2}\,\left(H(x+x_0)-H_L\right)=\sqrt{\frac{H'(x_0)}{H'(x+x_0)}}\,\left(H(x+x_0)-H_L\right)
\end{align*}
spans the above set.

To conclude we just need to check that $v^{[x_0]}\notin D_{x_0}$. This is indeed the case since
\begin{align*}
\frac{(v^{[x_0]})'}{v^{[x_0]}}(0)
&\,=\,-\frac12\frac{H''(x_0)}{H'(x_0)}+\frac{H'(x_0)}{H(x_0)-H_L}\\
&\,=\,\frac12 f_2(x_0)+\frac{(H(x_0)-H_R)(H(x_0)-H_{out})}{a(x_0)\,H(x_0)^2}\\
&\,\neq 
\,\frac12\,f_2(x_0)+\frac{(H(x_0)-H_L)(H(x_0)-H_{out})}{a(x_0)\,H(x_0)^2}
\,=\,c_2^{[x_0]}\,.
\end{align*}
\end{proof}

\subsection{The Riemann shock case}

We conclude our stability analysis by discussing how to adapt the above arguments to the Riemann shock case. The overall strategy is identical but details should be changed at various places.

We only indicate these modifications. To begin with, since $[r(W)]=0$, it is convenient to replace $\psi$ with $\tpsi:=\lambda\psi$. This does not change the nature of the spectral problem when $\lambda\neq0$ and simply decreases by $1$ the algebraic multiplicity of the eigenvalue $\lambda=0$. Our task is thus to determine when there exists a nonzero $(v,\tpsi)\in H^1_{\eta_L}(\R_-^*;\C^2)\times \C$ solving
\[
\left\{
\begin{array}{rl}
\lambda v+(A\,v)'&=\ E\,v,\qquad\textrm{on }\R_-^*,\\
\tpsi\left[\,W\right]
&=\ \left[Av\right]\,.
\end{array}
\right.
\]

Since $[W]$ is non zero, one may eliminate $\tpsi$ and reduce the discussion to the existence of a nonzero $u=(u_1,u_2):=(v_1,-cv_1+v_2)$ in $H^1_{\eta_L}(\R_-^*;\C^2)$ such that on $\R_-^*$
\[
\left\{
\begin{array}{l}
\lambda\,u_1+u_2'=0,\\
(au_1)'=\left(1-2\left(c-\frac{Q}{H}\right)\frac{Q}{H^2}-2\lambda\,\left(c-\frac{Q}{H}\right)\right)u_1-\left(\lambda+2\frac{Q}{H^2}+2\left(\frac{Q}{H}\right)'\right)\,u_2\,,
\end{array}
\right.
\]
and 
\[
u_2(0^-)\times\left[Q\right]
-\left(a\,u_1+\left(-c+2\frac{Q}{H}\right) \,u_2\right)(0^-)\times \left[H\right]\,=\,0\,.
\]
There is no nonzero solution with either $u_2$ vanishing identically or $\lambda=0$, so that the problem is equivalent to finding a nonzero $u_2\in H^2_{\eta_L}(\R_-^*;\C)$ solving on $\R_-^*$
\begin{align*}
u_2''+\left(f_1\lambda+f_2\right)u_2'+\lambda\,\left(f_3\lambda+f_4\right)\,u_2\,=\,0\,.
\end{align*}
and 
\[
u_2(0^-)\times\left[\lambda Q\right]
-\left(-a\,u_2'+\left(-c+2\frac{Q}{H}\right)\,\lambda\,u_2\right)(0^-)\times \left[ H\right]\,=\,0
\]
where
\begin{align*}
f_1&:=\frac{2}{a}\left(c-\frac{Q}{H}\right),&
f_2&:=-\frac{1}{a}\left(1-2\left(c-\frac{Q}{H}\right)\frac{Q}{H^2}\right),&
f_3&:=-\frac{1}{a},&f_4&:=-\frac{2}{a}\frac{Q}{H^2}\,,
\end{align*}
with $H$ and $Q$ constant equal to $H_L$ and $H_L^{3/2}$ respectively.

From here no change is needed in the reduction from $u_2$ to $w$, nor in Subsection~\ref{s:nonreal}. The core of Subsection~\ref{s:real} is simply replaced with a direct check that $\cB$ is positive definite. This follows from the fact that $f_2$ is a negatively-valued constant function and $c_2$ is also negative since it is equal to one half of this value. The sign observation stems from $H_L>H_s$ and $F>2$ which imply
\begin{align*}
1-2\left(c-\frac{Q_L}{H_L}\right)\frac{Q_L}{H_L^2}
\,=\,1-2\,\frac{q_0}{H_L^{\frac32}}
> 1-\frac{2}{F}>0\,.
\end{align*}

Summarizing the results of the present section with the ones of Section~\ref{s:spectral}, we obtain the following proposition.

\begin{proposition}[Convective exponential spectral stability in $\pazocal{R}_\mathrm{conv}$]
Discontinuous waves of region $\pazocal{R}_\mathrm{conv}$ are convectively exponentially spectrally stable. More precisely, when $\nu>1$ and $2\leq F <\sqrt{2\nu(1+\nu)}$, there is a choice of $\eta_L\in(\gamma_+^\infty,\gamma_-^\infty)$ and $\eta_R<0$ sufficiently negative such that the spectrum is included in $\left\{\lambda\in\C;\Re(\lambda)<-\theta<0\right\}\cup\left\{0\right\}$ in the $(\eta_L,\eta_R)$-weighted space for some $\theta>0$. Furthermore, $\lambda=0$ has multiplicity one. 
\end{proposition}

The above result is sharp since for $\nu>1$ and $F>\sqrt{2\nu(\nu+1)}$, that is in region $\pazocal{R}_\mathrm{abs}$, the corresponding waves are absolutely unstable as shown in Section~\ref{s:spectral}.

\section{Linear and nonlinear convective stability}\label{s:stab}
At this point, we have shown that convective spectral stability holds (with scalar weight)
for $F<\sqrt{2\nu(\nu+1)}$, and fails (for any weight) for $F>\sqrt{2\nu(\nu+1)}$, 
We now complete our discussion of convective stability by invoking a Lyapunov-type argument showing that convective spectral stability implies linear and nonlinear convective orbital stability, at time-exponential rate.

Convective spectral stability in the semilinear parabolic case, with a smooth background traveling wave, yields fairly immediately time-exponential asymptotic orbital stability, 
by well known arguments of Sattinger \cite{Sa} and Henry \cite{He} similar to those for the finite-dimensional ODE case. The present setting involving discontinuous background waves and quasilinear hyperbolic equations requires a much more technical analysis, at the frontier of the current knowledge on nonlinear wave stability theory.

The expository choice we make is to borrow results from the forthcoming \cite{FR2} that carries out a systematic development in a more general setting, in the spirit of \cite{FR1}. We stress however that, to a large extent, a relatively simple adaptation of the techniques used in \cite{YZ} for the neutrally stable case\footnote{See also the related discussion of \cite[p. 201]{YZ}.} would already be sufficient for the present case. Nevertheless a self-contained exposition of this adequate version would essentially double the size of the present contribution. Even for the smooth case, none of the results in the literature seems directly applicable but, likewise, a relatively simple variation on \cite{MZ,MaZ} would yield the required result.

\subsection{Linear estimates}\label{s:lin}

We begin with estimates for the linearized problem \eqref{affine-smooth}
\be\label{2affine-smooth}
\partial_t v+ \partial_x(A\,v) = E\,v+F,
\qquad\textrm{on }\R_+\times\R
\ee
with initial data $v_0(x)$ and interior source term $F(t,x)$. 

\begin{proposition}\label{linbd_smooth}
Let $W=(H,Q)$ be a traveling-wave solution of type \textup{(v)} and consider $(\eta_L,\eta_R)$ spatial weight growths ensuring\footnote{For instance, $\eta_L$ positive and $\eta_R$ negative both sufficiently small in absolute value.} a spectral gap. Then there exist positive $\theta$ and $C$ such that for any $1\leq p\leq \infty$, if
\[
(v_0,F)\in 
\begin{cases}
L^p_{\eta_L,\eta_R}(\R)\times \cC^0(\R_+;L^p_{\eta_L,\eta_R}(\R))&\quad 1\leq p<\infty\\
BUC^0_{\eta_L,\eta_R}(\R)\times \cC^0(\R_+;BUC^0_{\eta_L,\eta_R}(\R))&\quad p=\infty
\end{cases}
\]
then for $v$ the unique mild solution to \eqref{2affine-smooth} (in $\cC^0(\R_+;L^p_{\eta_L,\eta_R}(\R))$ if $1\leq p<\infty$, $\cC^0(\R_+;BUC^0_{\eta_L,\eta_R}(\R))$ if $p=\infty$) with initial data $v_0$, there exists a phase shift $\varphi \in\cC^1(\R_+)$ vanishing initially such that for any $t\geq0$
\ba\label{linbds_smooth}
\|v(t,\cdot)+\varphi(t)\,W'\|_{L^p_{\eta_L,\eta_R}}+ |\varphi'(t)|
	&\leq Ce^{-\theta t}\|v_0\|_{L^p_{\eta_L,\eta_R}} 
	+ C \int_0 ^t e^{-\theta (t-s)}\|F(s,\cdot)\|_{L^p_{\eta_L,\eta_R}}\md s\,.
	\ea
\end{proposition}

In the foregoing statement $BUC^0(\Omega)$ denotes the space of functions that are bounded on $\Omega$, and uniformly continuous on each connected component of $\Omega$.

We recall that Duhamel formula enables one to reduce the previous statement to the sourceless case. Moreover we point out that in the case $p=2$ the statement follows from the Gearhart-Pr\"uss theorem and high-frequency bounds on resolvents.

Consider again the linearized problem \eqref{affine-disc}:
\be\label{2affine-disc}
\left\{
\begin{array}{rl}
\partial_t v+ \partial_x(A\,v)&=\ E\,v+F,\qquad\textrm{on }\R_+\times\R^*,\\
\frac{\md \psi}{\md t}\left[\,W\,\right] -\psi\left[r(W)\right]
&=\ \left[Av\right]+G,\qquad\textrm{on }\R_+,
\end{array}
\right.
\ee
with initial data $(v_0(x),\psi_0)$, interior source term $F(t,x)$, and boundary source-term $G(t)$. Recall from the original derivation of \eqref{affine-disc} that here there is no freedom in the phase shift that may be removed from $v$ so as to obtain time decay. We need to prove that $v-(-\psi)\,W'$ is decaying. In contrast, in the smooth case, the phase shift $\varphi$ is far from unique.

\begin{proposition}\label{linbd_disc}
Let $W=(H,Q)$ be a traveling-wave solution of type \textup{(ii)}-\textup{(iv)} satisfying the sharp convective spectral stability condition\footnote{Automatically satisfied in cases (iii) and (iv).} 
\begin{align*}
F&<\sqrt{2\nu(\nu+1)}\,,& \nu&:=\sqrt{\frac{H_L}{H_R}}\,,
\end{align*}
and consider $(\eta_L,\eta_R)$ spatial weight growths ensuring\footnote{For instance, when $F>2$, $\eta_L\in(\eta_L^\mathrm{min},\eta_L^\mathrm{max})$ and $\eta_R$ sufficiently negative; when $F<2$, $\eta_L$ positive and $\eta_R$ negative both sufficiently small in absolute value.} a spectral gap. Then there exist positive $\theta$ and $C$ such that for any $1\leq p\leq \infty$, if
\[
(v_0,\psi_0,F,G)\in 
\begin{cases}
L^p_{\eta_L,\eta_R}(\R)\times\R\times \cC^0(\R_+;L^p_{\eta_L,\eta_R}(\R))\times \cC^0(\R_+)&\quad 1\leq p<\infty\\
BUC^0_{\eta_L,\eta_R}(\R^*)\times\R\times \cC^0(\R_+;BUC^0_{\eta_L,\eta_R}(\R^*))\times \cC^0(\R_+)&\quad p=\infty
\end{cases}
\]
then $(v,\psi)$, the unique mild solution to \eqref{2affine-disc} with initial data $(v_0,\psi_0)$, satisfies for any $t\geq0$
\ba\label{linbds_disc}
&\|v(t,\cdot)+\psi(t)\,W'\|_{L^p_{\eta_L,\eta_R}}+ |\psi'(t)|+\|\left(v(t,\cdot)+\psi(t)\,W'\right)(0^-)\|
+\|\left(v(t,\cdot)+\psi(t)\,W'\right)(0^+)\|\\
&\quad\leq Ce^{-\theta t}\|v_0\|_{L^p_{\eta_L,\eta_R}} 
+ C \int_0 ^t e^{-\theta (t-s)}\|F(s,\cdot)\|_{L^p_{\eta_L,\eta_R}}\md s
+ C \int_0 ^t e^{-\theta (t-s)}\|G(s)\|\,\md s\,.
\ea
\end{proposition}

Note that the level of regularity of the previous statement is insufficient, alone, to define traces at $0^\pm$. The existence of those is a consequence of the fact that $(v,\psi)$ solves \eqref{2affine-disc} and that the shock is non characteristic.

\subsection{Nonlinear stability}\label{s:nonlin}

Using Propositions~\ref{linbd_smooth} and~\ref{linbd_disc} in order to prove nonlinear stability results induces a severe loss of derivatives due to the quasilinear character of the original system. A by-now classical way to cure this loss is to combine the latter with nonlinear high-frequency damping estimates, that show that as long as the Lipschitz norm of the solution remains under control, the time decay of any Sobolev norm is slaved to the time decay of the $L^2$ norm. Designing such nonlinear high-frequency damping estimates is a significant part of the nonlinear stability analysis. When proceeding in this way, it is actually sufficient to prove linear stability with derivative losses, as accessible through what the fourth author has dubbed the ``poor man's Pr\"uss construction''  \cite{Z}. On nonlinear high-frequency damping estimates, we refer to \cite[Appendix~A]{R_HDR} for an introduction the classical Kawashima version for the stability of constant states \cite{Kaw,SK} and to \cite{MaZ,RZ,YZ,FR1} for versions more directly related to the present analysis.

\bt\label{main-smooth}
Let $W=(H,Q)$ be a traveling-wave solution of type \textup{(v)} and consider $(\eta_L,\eta_R)$ spatial weight growths ensuring\footnote{For instance, $\eta_L$ positive and $\eta_R$ negative both sufficiently small in absolute value.} a spectral gap, with $\eta_L$ postive and $\eta_R$ negative. Then there exist positive $\delta$, $\theta$ and $C$ such that if 
\[
\delta_{w_0}:=\|w_0-W\|_{H^2_{\eta_L,\eta_R}(\R)}\leq \delta
\]
then for $w$ the unique mild solution to \eqref{sv} (in $\cC^0(\R_+;H^2_{\eta_L,\eta_R}(\R))$) with initial data $w_0$, there exists a phase shift $\varphi \in\cC^1(\R_+)$ vanishing initially and an asymptotic shift $\varphi_\infty\in\R$ such that for any $t\geq0$
\ba\label{bds_smooth}
\|w(t,\cdot)-\,W(\cdot-(ct+\varphi(t)))\|_{H^2_{\eta_L,\eta_R}(\R)}+ |\varphi'(t)|
+|\varphi(t)-\varphi_\infty|
	&\leq Ce^{-\theta t}\delta_{w_0}\,,
\ea
and $|\varphi_\infty|\leq C\,\delta_{w_0}$.
\et

\bt\label{main-disc}
Let $W=(H,Q)$ be a traveling-wave solution of type \textup{(ii)}-\textup{(iv)} satisfying the sharp convective spectral stability condition\footnote{Automatically satisfied in cases (iii) and (iv).} 
\begin{align*}
F&<\sqrt{2\nu(\nu+1)}\,,& \nu&:=\sqrt{\frac{H_L}{H_R}}\,,
\end{align*}
and consider $(\eta_L,\eta_R)$ spatial weight growths ensuring\footnote{For instance, when $F>2$, $\eta_L\in(\eta_L^\mathrm{min},\eta_L^\mathrm{max})$ and $\eta_R$ sufficiently negative; when $F<2$, $\eta_L$ positive and $\eta_R$ negative both sufficiently small in absolute value.} a spectral gap, with $\eta_L$ postive and $\eta_R$ negative. Then there exist positive $\delta$, $\theta$ and $C$ such that if 
\[
\delta_{w_0}:=\|w_0-W\|_{H^2_{\eta_L,\eta_R}(\R^*)}\leq \delta
\]
with $w_0-W$ supported away from zero, then there exists a global solution to \eqref{sv}, $w$, with initial data $w_0$ having at each time $t\geq0$ a single shock, located at $ct+\psi(t)$, with $\psi \in\cC^1(\R_+)$ vanishing initially and an asymptotic shift $\psi_\infty\in\R$ such that for any $t\geq0$
\ba\label{bds_smooth}
\|w(t,\cdot+(ct+\psi(t)))-\,W\|_{H^2_{\eta_L,\eta_R}(\R^*)}+ |\psi'(t)|
+|\psi(t)-\psi_\infty|
	&\leq Ce^{-\theta t}\delta_{w_0}\,,
\ea
and $|\psi_\infty|\leq C\,\delta_{w_0}$.
\et

Note that none of the constants depend on how far the support of $w_0-W$ is from $0$. The assumption is simply made to assure that the initial data is compatible with the short-time persistence of a single-shock piecewise-$H^2$ solution. We could have instead assumed directly the optimal but cumbersome compatibility conditions, as in \cite{FR1,FR2}. On the related local-in-time propagation of regularity we refer to \cite{Me,BenzoniGavage-Serre_multiD_hyperbolic_PDEs}.

\section{Numerical time-evolution experiments}\label{s:num}

We augment our analytic treatment by examples of numerical time evolution experiments using CLAWPACK \cite{C1,C2}. More precisely, in a first set of numerical experiments, we test the convective nonlinear (in)stability of nonmonotone discontinuous waves of type (ii) depending if whether we are in region $\pazocal{R}_{\mathrm{conv}}$  or $\pazocal{R}_{\mathrm{abs}}$ ($(F,H_R/H_L)=(2.28,0.7)$ versus $(F,H_R/H_L)=(2.30,0.7)$ as an example), and also demonstrate the convective nonlinear stability of Riemann profiles (iii) ($(F,H_R/H_L)=(\sqrt{85}/10+\sqrt{238}/14,0.7)$ as an example). Finally, in a second set of numerical experiments, we highlight the convective nonlinear instability of increasing smooth waves of type (i) ($(F,H_R/H_L)=(3,1.3)$ as an example).
\begin{figure}[htbp]
    \centering
    \includegraphics[scale=0.32]{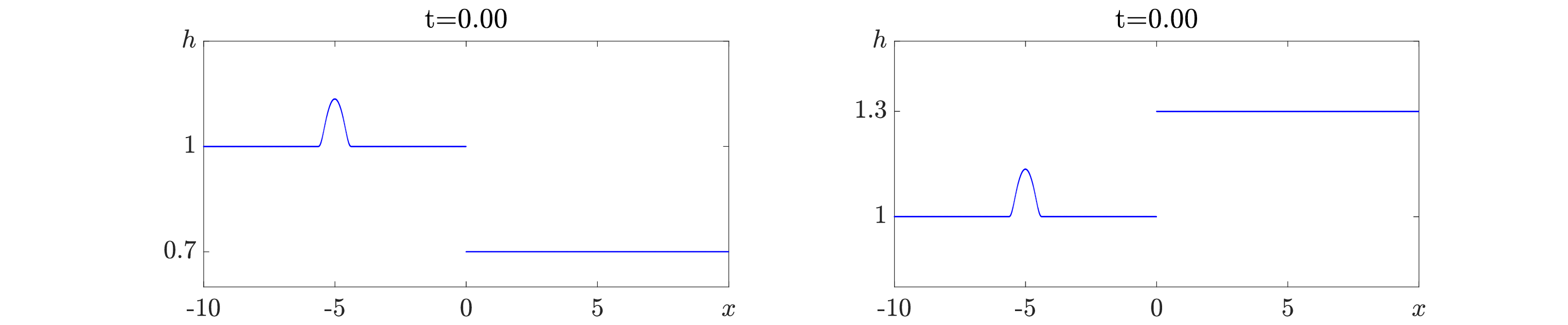}
    \caption{Left panel: Dambreak initial data \eqref{numIC1}; Right panel: Dambreak initial data \eqref{numIC2}.}
    \label{initial_data}
\end{figure}
\begin{figure}[htbp]
\includegraphics[scale=0.35]{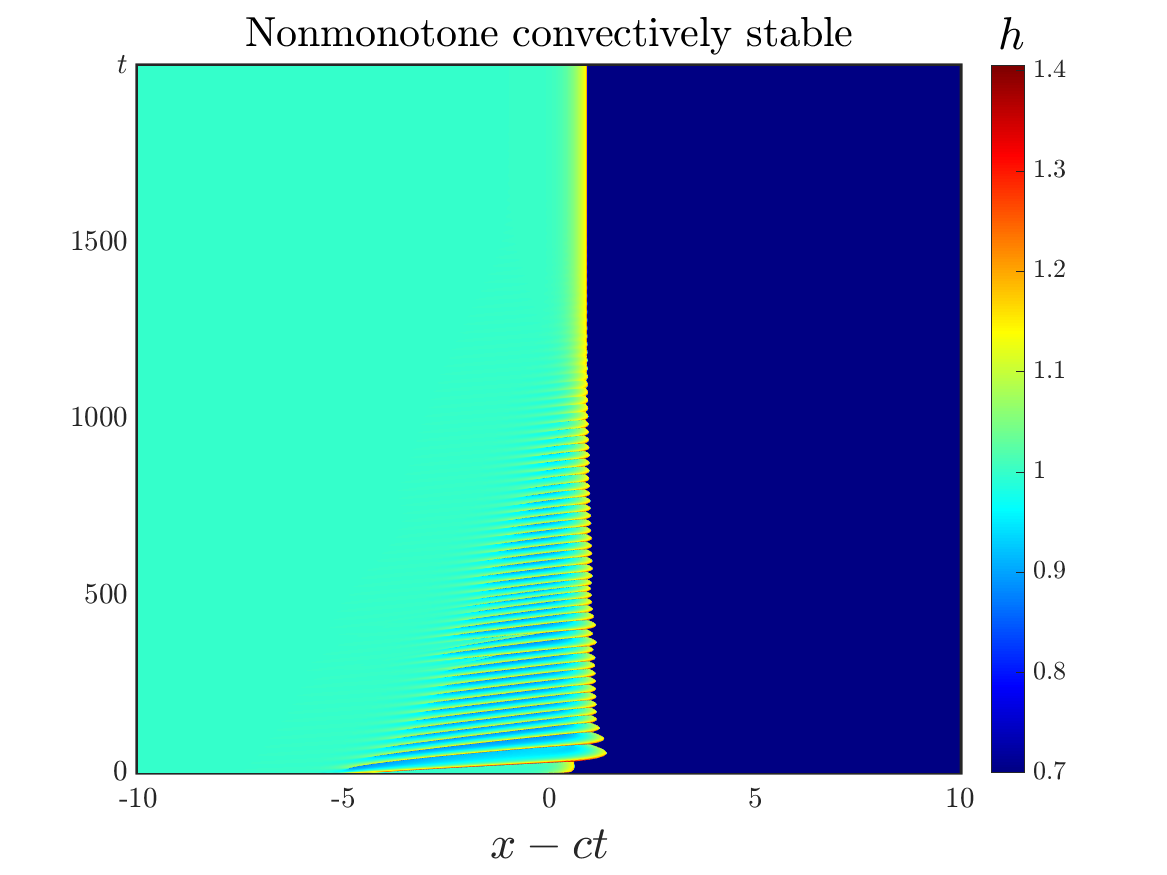}
\includegraphics[scale=0.35]{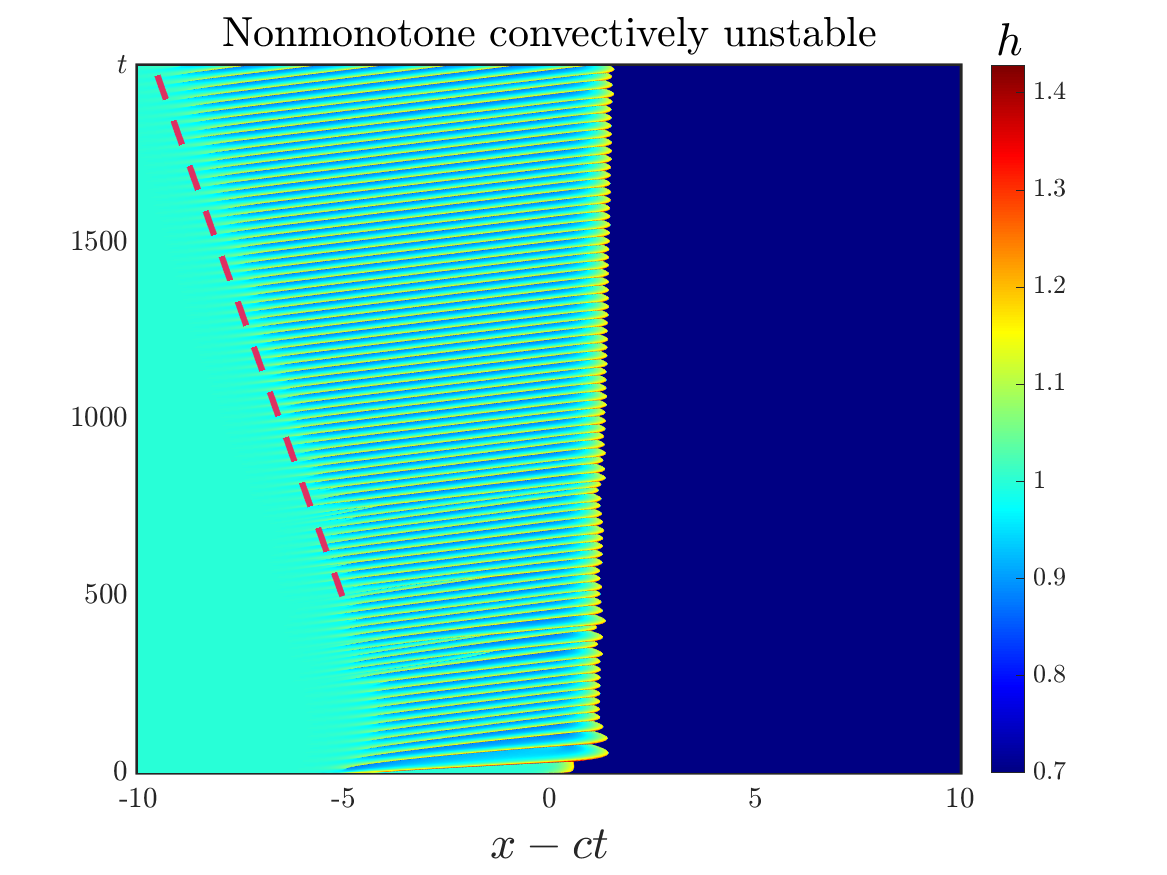}
\includegraphics[scale=0.35]{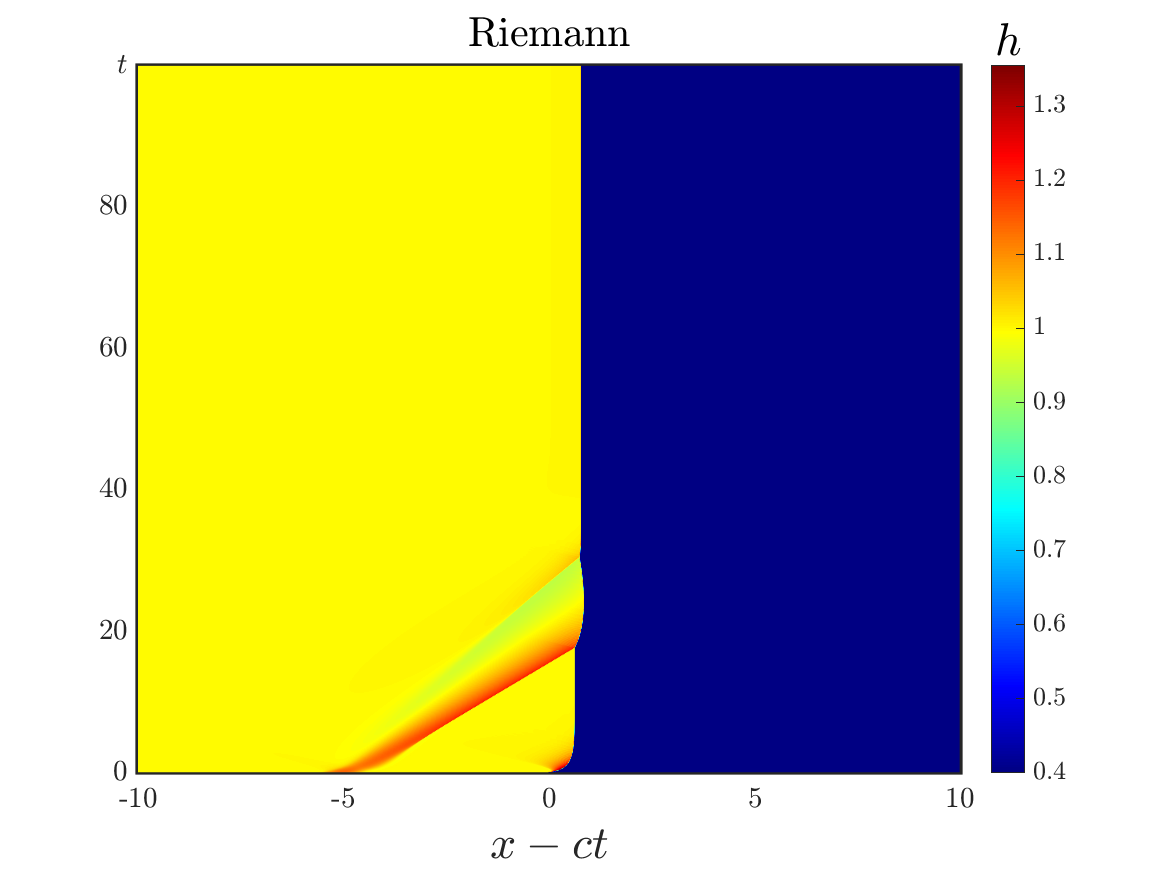}
\includegraphics[scale=0.35]{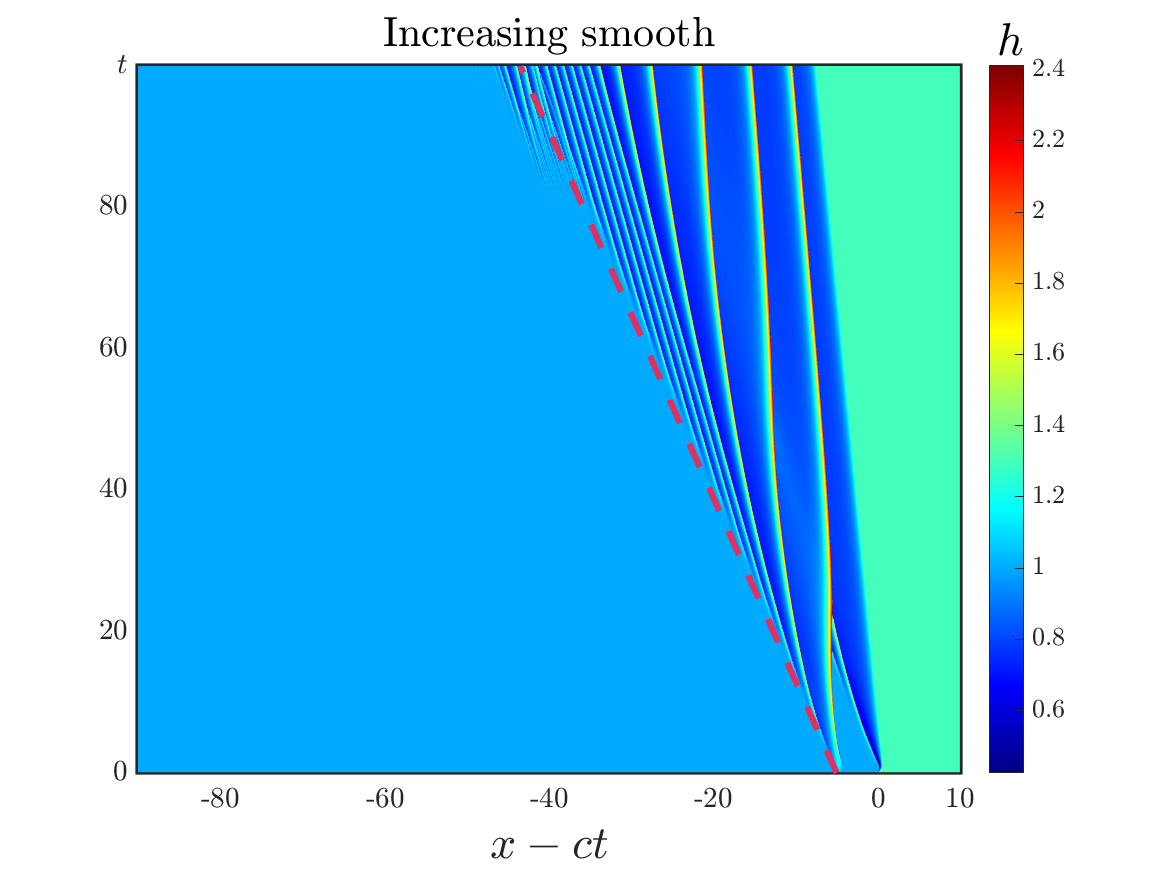}
\caption{Space-time plots of simulations shown in Figures~\ref{nonmonotone_stable}--\ref{increasing}.}
\label{space-time}
\end{figure}

\subsection{Nonmonotone discontinuous waves and Riemann shock} Throughout this section we set $H_L=1$ and $H_R=0.7$. We first note that profiles are of nonmonotone discontinuous type (ii) if $$2.02390\ldots=\frac{\sqrt{85}}{10}+\frac{\sqrt{238}}{14}<F<\frac{\sqrt{70}}{7} + \frac{10}{7}=2.62380\ldots$$ and they are convectively exponentially stable given that $$F<\sqrt{\frac{2}{0.7}+\frac{2}{\sqrt{0.7}}}= 2.29076\ldots.$$ When $F=\frac{\sqrt{85}}{10}+\frac{\sqrt{238}}{14}$ profiles are of Riemann shock type (iii) and convectively exponentially stable. 
In all cases, speed of the waves is given by
$$ c=\frac{H_L+\sqrt{H_LH_R}+H_R}{\sqrt{H_L}+\sqrt{H_R}}=1.3811\ldots.$$ 

For our numerical experiments, we use a perturbed dambreak initial data given by 
\begin{equation}
\label{numIC1}
\left\{
\begin{split}
h_0(x)&=\mathbbm{1}_{x\leq 0}+0.7\times\mathbbm{1}_{0<x}+\mathbbm{1}_{(x+5)^2<0.5}e^{-\frac{1}{0.5-(x+5)^2}},\\
q_0(x)&=\mathbbm{1}_{x\leq 0}+0.7^{3/2}\times\mathbbm{1}_{0<x},
\end{split}
\right. \quad x\in\R.
\end{equation}
See Figure~\ref{initial_data} left panel for a plot of \eqref{numIC1}.
\paragraph{{\bf Convectively stable regime}} For $F=2.28\in\pazocal{R}_{\mathrm{conv}}$, we present in Figure~\ref{nonmonotone_stable} several snapshots at time $100$, $500$, and $2000$ of the fluid height $h$ in the comoving frame $c$ showing convergence to a nonmonotone hydraulic shock in the large-time asymptotic limit. We also refer to Figure~\ref{space-time} first panel for a (comoving) space-time plot of the fluid height $h$. The latter space-time plot clearly shows the speeds of shocks are faster than the comoving frame speed and shocks gradually merge into a single subshock of the nomonotone hydraulic profile at $t\approx 1000$. This numerical experiment illustrates that our analytically derived convective stability condition predicts  the asymptotic response to large-scale localized perturbations: convergence to a nonmonotone hydraulic shock in our case.

\begin{figure}[htbp]
\begin{center}
\includegraphics[scale=0.4]{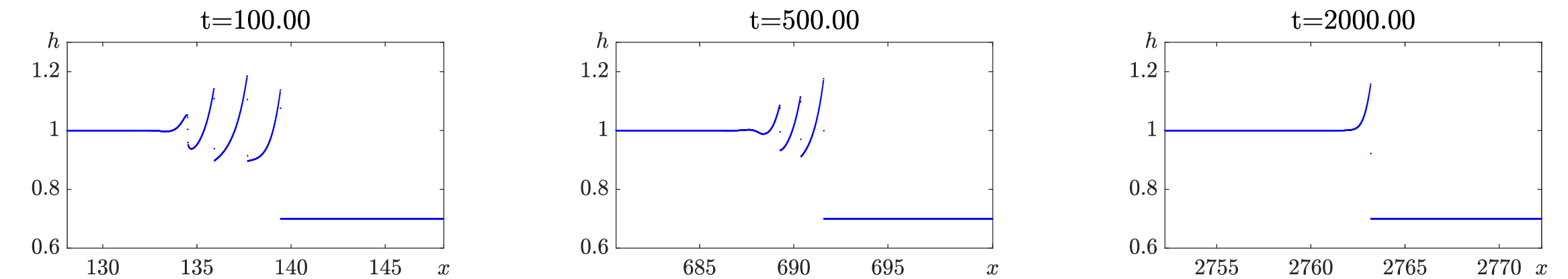}
\end{center}
\caption{Numerical simulation of \eqref{sv} with $F=2.28$ and initial data \eqref{numIC1}.}
\label{nonmonotone_stable}
\end{figure}

\paragraph{{\bf Convectively unstable regime}} On the other hand, for $F=2.30\in\pazocal{R}_{\mathrm{abs}}$ in the convectively unstable regime, we make a simulation with the former initial data, showing an
“invading front” connecting roll wave patterns on the left to a constant state on the right. See Figure~\ref{nonmonotone_unstable} for plots at time $100$, $500$, and $2000$ of fluid heights in the comoving frame $c$ and Figure~\ref{space-time} second panel for a (comoving) space-time plot. The latter space-time plot clearly shows that although the speeds of shocks are faster than that of the comoving frame, the location where new shocks emerge (marked by a dash line in Figure~\ref{space-time} second panel) moves at a slower speed than the comoving frame speed, resulting in more and more shocks between the location where new shocks emerge and the last shock connecting to $H_R$. 

Note that the slower speed of the invading front can be heuristically predicted by tracking how the absolute spectrum associated with $H_L$ depends on the speed of the moving frame in which it is computed. To be more concrete, we revisit computations from Section~\ref{s:consist} by allowing the speed of the moving frame $s$ to vary (instead of being fixed to $c$, the wave speed of the traveling wave of interest) and correspondingly mark with $s$ different quantities introduced there. We are interested in the absolute spectrum of $H_L$ at speed $s$, that is in the $\lambda$ such that the real parts of the spatial eigenvalues of $G_{H_L,s}(\lambda)$ coincide, $\Re(\gamma_{-,H_L,s}(\lambda))=\Re(\gamma_{+,H_L,s}(\lambda))$. Following \cite{FHSS}, we may define a so-called absolute spreading speed $s_\mathrm{abs}$ as the infimum of wave speeds $s$ for which the absolute spectrum remains unstable in the moving frame $s$. Computations from Section~\ref{s:consist} yield
\[
s_{abs}=\sqrt{H_L}\left(1+\frac{2}{F^2}\right).
\]
In our case, with $F=2.30$, $H_L=1$ and $H_R=0.7$, we have $s_\mathrm{abs}-c=-0.00305\ldots$. The dashed line in the second panel of Figure~\ref{space-time} has precisely a slope given by $s_\mathrm{abs}-c$ which matches the onset of the invading roll waves quite accurately.

\begin{figure}[htbp]
\begin{center}
\includegraphics[scale=0.4]{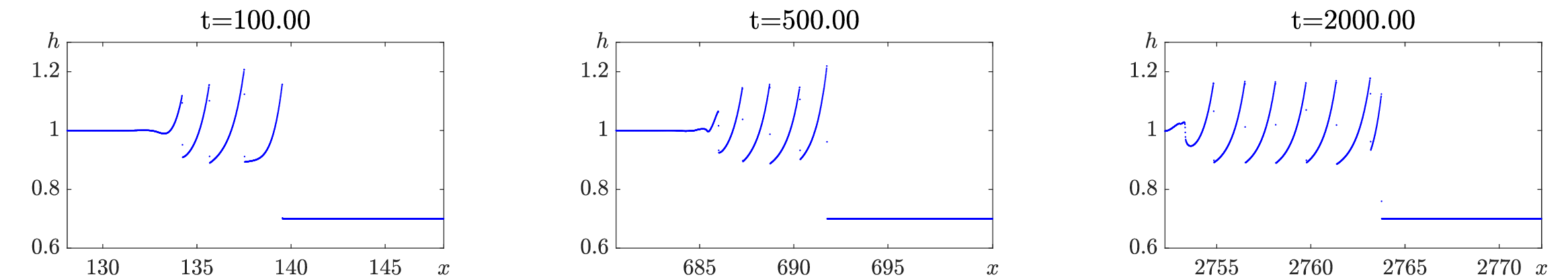}
\end{center}
\caption{Numerical simulation of \eqref{sv} with $F=2.30$ and initial data \eqref{numIC1}.}
\label{nonmonotone_unstable}
\end{figure}

\paragraph{{\bf Riemann case}} Finally, for $F=\sqrt{85}/10+\sqrt{238}/14$ in domain (iii) of Riemann profiles which lies within the convectively stable regime $\pazocal{R}_{\mathrm{conv}}$, we simulate \eqref{sv} with the initial data \eqref{numIC1}, showing convergence to the unperturbed Riemann shock in the large-time asymptotic limit. See Figure~\ref{Riemann} for plots at time $10$, $50$, and $100$ of fluid heights in the comoving frame $c$ and Figure~\ref{space-time} third panel for a (comoving) space-time plot. Both plots show emergence of a single shock caused by the initial perturbation which quickly merges into the Riemann shock.

\begin{figure}[htbp]
\begin{center}
\includegraphics[scale=0.4]{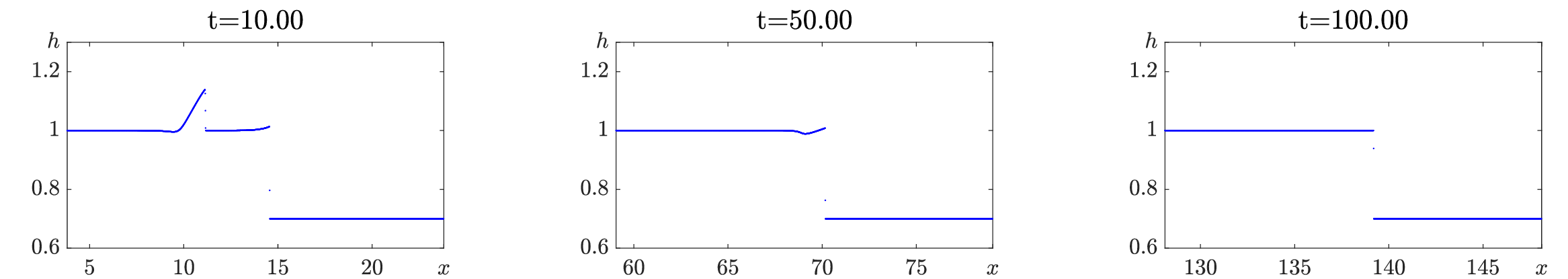}
\end{center}
\caption{Numerical simulation of \eqref{sv} with $F=\frac{\sqrt{85}}{10}+\frac{\sqrt{238}}{14}$ and initial data \eqref{numIC1}.}
\label{Riemann}
\end{figure}

\subsection{Increasing smooth waves}

Finally, we make $H_R>H_L$ to test if the corresponding increasing smooth ``reverse shock'' can be the large-time asymptotic limit. We fix $H_L=1$ and $H_R=1.3$ such that profiles are of increasing smooth type (i) if 
$$
F>\sqrt{\frac{13}{10}} + \frac{13}{10}=2.44017\ldots.
$$
For $F=3\in\pazocal{R}_{\mathrm{abs}}$ in domain (i) of increasing smooth profiles, we simulate with dambreak initial data given by 
\begin{equation}
\label{numIC2}
\left\{
\begin{split}
h_0(x)&=\mathbbm{1}_{x\leq 0}+1.3\times\mathbbm{1}_{0<x}+\mathbbm{1}_{(x+5)^2<0.5}e^{-\frac{1}{0.5-(x+5)^2}},\\
q_0(x)&=\mathbbm{1}_{x\leq 0}+1.3^{3/2}\times\mathbbm{1}_{0<x},
\end{split}
\right. \quad x\in\R.
\end{equation}
See Figure~\ref{initial_data} right panel for a plot of \eqref{numIC2}.
We report an “invading back” connecting roll wave patterns on the left to a constant state on the right. See Figure~\ref{increasing} for plots at time $5$, $25$, and $100$ of fluid heights in the comoving frame $c=\frac{H_L+\sqrt{H_LH_R}+H_R}{\sqrt{H_L}+\sqrt{H_R}}\sim1.6074$ and Figure~\ref{space-time} last panel for a space-time plot. This illustrates once again that our (local) stability conditions indeed successfully predict large-scale asymptotic behavior.

We also tested the predictive feature of the absolute spreading speed introduced in the convectively unstable case beyond its expected range of validity by computing $s_\mathrm{abs}$ in the present case and found $s_\mathrm{abs}-c=-0.3852\ldots$. Quite surprisingly and remarkably,  this predicted speed compares well with the speed of the primary invading front (see the dashed line in the forth panel of Figure~\ref{space-time}) for short time.


\begin{figure}[htbp]
\begin{center}
\includegraphics[scale=0.4]{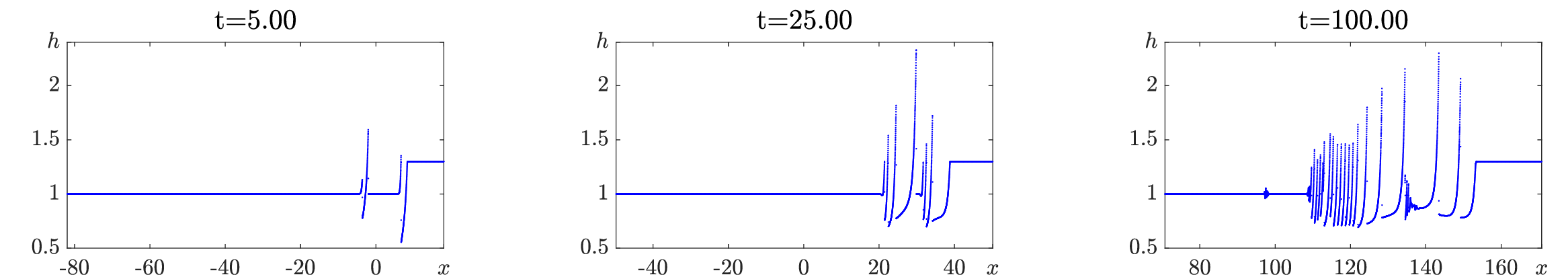}
\end{center}
\caption{Numerical simulation of \eqref{sv} with $F=3$ and initial data \eqref{numIC2}.}
\label{increasing}
\end{figure}

\paragraph{{\bf Roll wave selection}}

So far we are lacking even a heuristic argument to predict which roll wave is selected in the invading front pattern. Let us recall that, even when the translational invariance is quotiented, roll waves form a two-parameter family. 

By integrating over a large space-time box the conservation law of (SV), one may derive formally a constraint equation 
\[
q_{0,shock} = q_{0,{roll}}+ (c_{roll}-c_{shock})\langle H_{roll}\rangle,
\]
where $\langle A\rangle$ denotes the average of the quantity $A$ over one period of the roll wave pattern. Yet this is one equation short to fully identify the roll pattern, leaving a degree of freedom still to be determined.

Looking toward the future, we would like to add two more comments on this question. First, we point out that we have estimated numerically the wave period and wave speed of the observed roll pattern  (as approximately $1$ and $1.4$ respectively) and checked that the wave does lie in the stability region of the roll-wave stability diagram \cite[Fig. 3(c)]{JNRYZ}.

Second, we mention that the oscillatory instability pattern between $H_L$ and the roll-wave seems to be expanding linearly in time, preventing a direct connection from the roll-wave to $H_L$. One can not exclude that the identification of the missing roll parameter requires a deep understanding of this pattern in a way reminiscent of the resolution of the Gurevich--Pitaevskii problem through the analysis of dispersive shocks \cite{BMR-DSW}.

A complete, rigorous treatment of this bifurcation would be very interesting to carry out.



\bibliographystyle{alphaabbr}
\bibliography{Ref-roll}

\end{document}